\documentclass[12pt]{article}
\usepackage{graphicx}
\usepackage[all]{xy}
\usepackage{color}
\usepackage{latexsym}
\usepackage{amssymb,amsmath,amstext}
\usepackage{epsfig}
\usepackage{graphics}
\usepackage{euscript}
\usepackage{ifthen}
\usepackage{wrapfig}
\usepackage{amsthm}
\usepackage{multicol}
\usepackage{paralist}
\usepackage{verbatim}
\usepackage[colorlinks=true, urlcolor=blue, linkcolor=blue, citecolor=blue]{hyperref}
\usepackage[margin=1in]{geometry}
\usepackage{algorithmicx}
\usepackage{algpseudocode}
\usepackage{algorithm}
\usepackage{placeins}
\usepackage{caption}
\usepackage{subcaption}
\usepackage{longtable}
\usepackage{lscape}
\usepackage{fancyhdr}
\usepackage{hyperref}
\usepackage{listings}

\theoremstyle{plain}
\newtheorem{theorem}{Theorem}[section]
\newtheorem{corollary}[theorem]{Corollary}
\newtheorem{lemma}[theorem]{Lemma}
\newtheorem{proposition}[theorem]{Proposition}

\theoremstyle{definition}
\newtheorem{definition}[theorem]{Definition}
\newtheorem{remark}[theorem]{Remark}
\newtheorem{example}[theorem]{Example}
\newtheorem{question}[theorem]{Question}

\numberwithin{equation}{section}

\newcommand{\R}{\mathbb{R}}

\newcommand{\Krank}{\mathrm{K\text{-}rank}}

\newcommand{\B}[1]{\mathbb #1}
\newcommand{\C}[1]{\mathcal #1}

\DeclareMathOperator{\interior}{int}

\DeclareMathOperator{\rank}{rank}

\DeclareMathOperator{\cone}{cone}
\DeclareMathOperator{\Orth}{O}
\newcommand{\nrank}{\rank_+}

\DeclareMathOperator{\cprank}{cp\text{-}rank}
\newcommand{\inner}[1]{\langle #1 \rangle}
\DeclareMathOperator{\im}{im}

\allowdisplaybreaks

\date{\today}
\begin{document}
\title{Uniqueness of nonnegative matrix factorizations by rigidity theory}
\author{Robert Krone and Kaie Kubjas}

\maketitle

\begin{abstract}
Nonnegative matrix factorizations are often encountered in data mining applications where they are used to explain datasets by a small number of parts. For many of these applications it is desirable that there exists a unique nonnegative matrix factorization up to trivial modifications given by scalings and permutations. This means that model parameters are uniquely identifiable from the data.  Rigidity theory of bar and joint frameworks is a field that studies uniqueness of point configurations given some of the pairwise distances. The goal of this paper is to use ideas from rigidity theory to study uniqueness of nonnegative matrix factorizations in the case when nonnegative rank of a matrix is equal to its rank. We characterize infinitesimally rigid nonnegative factorizations, prove that a nonnegative factorization is infinitesimally rigid if and only if it is locally rigid and a certain matrix achieves its maximal possible Kruskal rank, and show that locally rigid nonnegative factorizations can be extended to globally rigid nonnegative factorizations. These results give so far the strongest necessary condition for the uniqueness of a nonnegative factorization. We also explore connections between rigidity of nonnegative factorizations and boundaries of the set of matrices of fixed nonnegative rank. Finally we extend these results from nonnegative factorizations to completely positive factorizations.
\end{abstract}

\section{Introduction}

Nonnegative matrix factorization of size $r$ decomposes a matrix $M \in \mathbb{R}^{m \times n}_{\geq 0}$ as $M=AB$ where $A \in \mathbb{R}^{m \times r}_{\geq 0}$ and $B \in \mathbb{R}^{r \times n}_{\geq 0}$. The smallest $r \in \mathbb{N}$ such that $M$ has a size-$r$ nonnegative factorization is called the nonnegative rank of $M$.  Approximations by matrices of low nonnegative rank are ubiquitous in data mining applications where they are used to explain a dataset by a small number of parts; the number of parts being equal to nonnegative rank of the approximation. For example, Lee and Seung~\cite{lee1999learning} used nonnegative matrix factorizations for studying databases of face images. In this application, rows of $M$ correspond to different pixels of an image and columns of $M$ correspond to different images. A size-$r$ nonnegative factorization finds $r$ basis images (corresponding to columns of $A$) such that every original image is a nonnegative linear combination of these basis images (nonnegative coefficients are given by columns of $B$). Another popular application is topic modeling~\cite{xu2003document}, where the matrix $M$ gives frequencies of words in documents, and a nonnegative matrix factorization decomposes this matrix with respect to topics. Nonnegative matrix factorizations and nonnegative rank appear also in complexity theory~\cite{yannakakis1991expressing}, computational biology~\cite{devarajan2008nonnegative}, music analysis~\cite{smaragdis2003non}, blind source separation~\cite{cichocki2009nonnegative}, spectral data analysis~\cite{pauca2006nonnegative}.  For many of these applications it is desirable that there exists an essentially unique nonnegative factorization that explains the data, i.e. model parameters are identifiable from the data. We say essentially unique because nonnegative matrix factorizations are never completely unique: Given a nonnegative factorization $AB$ one obtains a new factorization $(AC)(C^{-1}B)$ by multiplying $A$ and $B$ by a scaling or permutation matrix and its inverse correspondingly.

The uniqueness of nonnegative factorizations was first addressed by Donoho and Stodden~\cite{donoho2004does} for black and white images with $P$ parts such that each part can appear in $A$ articulations. They showed that separability and complete factorial sampling guarantee uniqueness of nonnegative matrix factorization. Separability requires that one of the factors contains the $r \times r$ identity matrix as a submatrix and complete factorial sampling requires that the database contains all $A^P$ images where each of the $P$ parts appears in each of the $A$ articulations. Another sufficient condition appears in the work of Gillis~\cite{gillis2012sparse} and it requires $M$ to have $r$ nonzero columns each with $r-1$ zero entries with different sparsity patterns. Theis et al~\cite{theis2005first} prove uniqueness under a sparsity assumption on the nonnegative factors. Ding et al show that nonnegative matrix factorizations are unique assuming that one of the factors is orthogonal~\cite{ding2006orthogonal}. Many authors have established guarantees for identifiability under volume minimization or maximization of the polytope associated to one of the factors~\cite{chan2009convex,wang2009nonnegative,
fu2015blind,lin2015identifiability,fu2018identifiability}.  The first necessary condition was given by Laurberg et al~\cite{laurberg2008theorems} and it requires the rows of $A$ and columns of $B$ to be boundary closed. More precisely, for every $i \neq j \in [r]$ there must exist a row $a_k$ of $A$ such that $a_{ki}=0$ and $a_{kj} \neq 0$ (and similarly for columns of $B$).  A comprehensive review on uniqueness of nonnegative matrix factorizations is given by Fu et al~\cite{fu2019nonnegative}. Despite the recent progress on uniqueness of nonnegative matrix factorizations, the current sufficient conditions are either relatively restrictive or require additional assumptions on nonnegative factors, and a little is known about necessary conditions.

The goal of this paper is to study the uniqueness of nonnegative matrix factorizations by building on the rigidity theory of bar and joint frameworks, which studies uniqueness of point configurations given some pairwise distances between the points.  This approach has been already successfully adapted to investigating the uniqueness of low-rank matrix completion~\cite{singer2010uniqueness}. Similarly to rigidity theory, we define infinitesimally, locally and globally rigid nonnegative matrix factorizations. We consider the case when nonnegative rank is equal to rank. Before going into more details, we give a brief overview of the implications between these notions. Global rigidity is the same as the uniqueness of a nonnegative matrix factorization. Local rigidity is a necessary condition for global rigidity and infinitesimal rigidity is a sufficient condition for local rigidity. We give a characterization of infinitesimal rigidity that can be checked computationally (Proposition~\ref{theorem:infinitesimally_rigid_factorizations}). We show that infinitesimal rigidity implies local rigidity and that a locally rigid nonnegative matrix factorization that is not infinitesimally rigid implies that the Kruskal rank of a specified matrix is not maximal (Proposition~\ref{prop:generic_locally_rigid_is_infinitesimally_rigid}). These results lead to Algorithm~\ref{algorithm:local_rigidity} for determining local rigidity of a nonnegative matrix factorization (one possible output of the algorithm is that local rigidity of the matrix cannot be determined) and to a necessary condition for uniqueness of nonnegative factorizations that strengthens the necessary condition in~\cite[Theorem 3]{laurberg2008theorems} (Corollary~\ref{corollary:necessary_condition_for_global_rigidity_combinatorial}). A next step will be to use rigidity theory to study sufficient conditions for global rigidity. To do this, we believe that  one has to come up with an analogue of a stress matrix in rigidity theory, similarly as we have developed an analogue of a rigidity matrix in this work.

In more detail, in the rigidity theory of frameworks and low-rank matrix completion, infinitesimal motions are required to have derivatives of pairwise distances or inner products equal to zero. A framework is called infinitesimally rigid if all its infinitesimal motions are trivial ones. Checking infinitesimal rigidity is equivalent to checking rank of a rigidity or a completion matrix. In the nonnegative matrix factorization case, infinitesimal motions are additionally required to preserve nonnegativity of factors. Now checking infinitesimal rigidity amounts to checking whether positive span of a matrix is isomorphic to a specified linear subspace of $\mathbb{R}^{r^2}$ (Proposition~\ref{theorem:infinitesimally_rigid_factorizations}). The difference with the frameworks and low-rank matrix completion case is that instead of linear span one has to consider positive span of a rigidity matrix. Hence a linear algebra problem becomes a polyhedral geometry problem. We also give purely combinatorial necessary conditions for infinitesimal rigidity that follow from this characterization (Theorem~\ref{prop:Wpoint}, Lemmas~\ref{lemma:max_number_entries} and~\ref{lemma:zero rectangles}). 
Infinitesimal rigidity always implies local rigidity, and although the converse is not always true as we will see in Example~\ref{example:Shitov}, then if a nonnegative factorization is locally rigid and a certain matrix achieves its maximal possible Kruskal rank, then it is infinitesimally rigid  (Proposition~\ref{prop:generic_locally_rigid_is_infinitesimally_rigid}).  We also show that every locally rigid nonnegative factorization can be extended to globally rigid nonnegative factorization by adding at most $r$ strictly positive rows to $A$ and at most $r$ strictly positive columns to $B$ (Corollary~\ref{cor:locally_rigid_to_globally_rigid}).

Matrices  of size $m \times n$  and nonnegative rank at most $r$ form a semialgebraic set, which we denote by $\mathcal{M}^{m \times n}_{\leq r}$.  
We explore connections between rigidity of nonnegative matrix factorizations and boundaries of the set $\mathcal{M}^{m \times n}_{\leq r}$. The first motivation for this is that a matrix with a unique nonnegative matrix factorization always lies on the boundary of $\mathcal{M}^{m \times n}_{\leq r}$. The second motivation is that understanding boundaries of a semialgebraic set is often easier than deriving a semialgebraic description of the set, and sometimes boundaries provide the first step towards obtaining a semialgebraic description. This was the case for matrices of nonnegative rank at most three~\cite{kubjas2015fixed}. This semialgebraic description gives an algorithm, polynomial in $m$ and $n$, to decide if a rank-three matrix has nonnegative rank three by checking one condition for each possible boundary component.  A semialgebraic description of the set $\mathcal{M}^{m \times n}_{\leq r}$ would in general allow one to check directly whether a matrix has nonnegative rank at most $r$ without constructing a nonnegative factorization of the matrix. Neither boundaries nor a semialgebraic description of $\mathcal{M}^{m \times n}_{\leq r}$ is known for $r \geq 4$. Vavasis showed that computing nonnegative rank is NP-hard~\cite{vavasis2009on} and the best known algorithm for deciding whether an $m\times n$ matrix has nonnegative rank at most $r$ runs in time $(mn)^{O(r^2)}$ by the work of Moitra~\cite{moitra2016an}. A necessary and sufficient condition for a matrix to lie on the boundary of $\mathcal{M}^{m \times n}_{\leq 3}$ is that it contains a zero or all its size three nonnegative factorizations are infinitesimally rigid \cite{mond2003stochastic}.  This is not true for $r>3$. Example~\ref{example:Shitov}  provides a nonnegative matrix factorization that is locally and globally rigid, and hence on the boundary, but not infinitesimally rigid. Furthermore, in Section~\ref{sec:MSvS_example} we will see matrices on the boundary of $\mathcal{M}^{m \times n}_{\leq r}$ with nonnegative factorizations that are not even locally rigid.

We finish the paper with extending our results to completely positive factorizations. Let $M$ be a nonnegative real symmetric matrix. The completely positive rank of $M$ is the smallest $r$ such that $M=AA^T$ for some nonnegative $n\times r$ matrix $A$~\cite{abraham2003completely}. We consider real symmetric matrices whose completely positive rank is equal to their rank. We define infinitesimally, locally and globally rigid completely positive factorizations, and show that results analogous to the nonnegative factorizations case hold.

The outline of our paper is the following. In Section~\ref{sec:preliminaries}, we give preliminaries on rigidity theory (Section~\ref{sec:rigidity_theory}),  geometric characterizations of nonnegative rank via nested polytopes (Section~\ref{sec:geometric_characterization}) and nonnegative rank boundaries (Section~\ref{sec:boundaries}). In Section~\ref{sec:rigid}, we study infinitesimally rigid factorizations.  In Section~\ref{sec:isolated}, we study locally rigid nonnegative factorizations.   In Section~\ref{sec:rigidity_and_boundaries}, we study connections between rigidity and boundaries of $\mathcal{M}^{m \times n}_{\leq r}$.  In Section \ref{sec:symmetric}, we adapt these results on nonnegative rank of general matrices, to the case of completely positive rank on symmetric matrices. In Appendix~\ref{sec:realizations_of_infinitesimally_rigid_factorizations}, we show that in the case of $5 \times 5$ matrices of nonnegative rank four, for every zero pattern that satisfies the necessary condition in Theorem~\ref{prop:Wpoint}, there exists an infinitesimally rigid nonnegative factorization $(A,B)$ that realizes the zero pattern.  Code for computations in this paper is available at
\begin{center}
\url{https://github.com/kaiekubjas/nonnegative-rank-four-boundaries}
\end{center}

\section{Preliminaries} \label{sec:preliminaries}

\subsection{Rigidity theory} \label{sec:rigidity_theory}

The goal of rigidity theory is to determine whether $n$ points in $\mathbb{R}^d$ can be determined uniquely up to rigid transformations (translations, rotations, reflections) given a partial set of pairwise distances between them. We will introduce rigidity theory following~\cite[Section 2]{singer2010uniqueness} and discuss the connection between the rigidity theory and uniqueness of low-rank matrix completions established by Singer and Cucuringu~\cite[Sections 3 and 4]{singer2010uniqueness}. This subsection can be skipped at the first reading and used as a reference.

A bar and joint framework $G(p)$ in $\mathbb{R}^d$ consists of a graph $G=(V,E)$, a set of distances $\{d_{ij} \in \mathbb{R}_{\geq 0}:(i,j) \in E\}$ and a set of points $p_1,\ldots,p_{|V|} \in \mathbb{R}^d$ such that $\|p_i-p_j\| = d_{ij}$ for all $(i,j) \in E$. One can think of the distance constraints as bars that are joining corresponding points. Consider a motion of the bar and joint framework parametrized by $t$, i.e. $p_i(t)$ is the position of the $i$-th point at time $t$. To preserve the distances given by $E$, the motion has to satisfy
$$
\frac{d}{dt} \|p_i-p_j\|^2 = 0 \text{ for all } (i,j) \in E.
$$
Denoting the velocity of $p_i$ by $\dot{p}_i$ for $i=1,\ldots,|V|$, these constraints can be rewritten as
\begin{equation} \label{eqn:infinitesimally rigid}
(p_i-p_j)^T(\dot{p}_i - \dot{p}_j)=0 \text{ for all } (i,j) \in E,
\end{equation}
or in the matrix form as $R_G(p) \dot{p}=0$ where $R_G(p)$ is a $|E| \times n|V|$ matrix and $\dot p = (\dot{p}^T_1,\ldots,\dot{p}^T_n)^T$. The matrix $R_G(p)$ is called the rigidity matrix of the bar and joint framework.

A motion satisfying Equation~(\ref{eqn:infinitesimally rigid}) is called an infinitesimal motion. Trivial motions are motions given by rotation and translation of the entire framework, also referred to as rigid transformations.  A trivial motion satisfies $\dot{p}_i = D p_i + b$  with $D \in \mathbb{R}^{d \times d}$ skew-symmetric and $b \in \mathbb{R}^d$, and every trivial motion is infinitesimal. A bar and joint framework is called infinitesimally rigid if all its infinitesimal motions are trivial. There are $\frac{(d-1)d}{2}$ degrees of freedom choosing a skew-symmetric matrix $D$ (rotations) and $d$ degrees of freedom choosing a vector $b$ (translations).  Every trivial motion is in the kernel of the rigidity matrix $R_G(p)$, so the framework is infinitesimally rigid if and only if the dimension of the kernel of the rigidity matrix $R_G(p)$ is equal to $\frac{d(d+1)}{2}$.

A framework $G(p)$ is locally rigid if there exists a neighborhood $\mathcal{N}$ of the framework $G(p)$ such that $G(p)$ is the only framework up to rigid transformations with the same distance constraints in the neighborhood $\mathcal{N}$. A framework $G(p)$ is regular  if $\rank R_G(p) = \max \{ \rank R_G(q):\|q_i-q_j\| = d_{ij} \text{ for all } (i,j) \in E\}$. 

\begin{theorem}[\cite{asimow1979rigidity}] \label{thm:Asimow_Roth}
A framework is infinitesimally rigid if and only if it is regular and locally rigid.
\end{theorem}

A framework is called generic if the coordinates of the points $p_1,\ldots,p_{|V|}$ are algebraically independent over $\mathbb{Q}$. Any generic framework is regular. Theorem~\ref{thm:Asimow_Roth} implies that local rigidity is a generic property in the sense that if generic $G(p)$ is locally rigid then most frameworks $G(q)$ are locally rigid. Hence one can talk about local rigidity of graphs. This also allows one to check with probability one whether a framework is locally rigid by choosing a random configuration $p_1,\ldots,p_{|V|}$ and checking whether the dimension of the kernel of the rigidity matrix $R_G(p)$ is equal to $\frac{d(d+1)}{2}$.

Finally, a framework $G(p)$ is globally rigid if all other frameworks in $\mathbb{R}^d$ that have the same distance constraints are related to $G(p)$ by rigid transformations. Global rigidity is also a generic property, and there are necessary and sufficient results using ranks of stress matrices for checking generic global rigidity. However, since we focus on infinitesimal and local rigidity of nonnegative factorizations in this paper, we do not present them here.

Singer and Cucuringu established a connection between the rigidity theory and low-rank matrix completion~\cite{singer2010uniqueness}. Let $M$ be a $m \times n$ matrix of rank $r$ and let $(A,B)$ give a rank-$r$ factorization of $M$. Let  the rows of $A$ be $a^T_1,\ldots,a^T_m \in \mathbb{R}^r$ and the columns of $B$ be $b_1,\ldots,b_n \in \mathbb{R}^r$. Then $M_{ij} = a^T_i b_j$. 

The observed entries of $M$ define a bipartite graph $G=(V,E)$ on $m+n$ vertices. The vertices $V$ correspond to $a_1,\ldots,a_m,b_1,\ldots,b_n$ and the edges $E$ correspond to observed entries of $M$. Instead of distance constraints, one fixes inner products $M_{ij}=a^T_i b_j$ for $(i,j) \in E$. The graph $G$, the inner products $\{M_{ij} \in \mathbb{R}:(i,j) \in E\}$ and the points $a_1,\ldots,a_m,b_1,\ldots,b_n \in \mathbb{R}^r$ define a framework. Consider a deformation of a framework parametrized by $t$. To preserve the inner products $M_{ij}=a^T_i b_j$ for $(i,j) \in E$, the deformation has to satisfy
\begin{equation} \label{eqn:matrix_infinitesimally_rigid}
a^T_i \dot{b}_j + \dot{a}^T_i b_j = 0 \text{ for all } (i,j) \in E
\end{equation}
where $\dot{a}$ and $\dot{b}$ are velocities of $a$ and $b$. The same constraints can be written in a matrix form using the $r \times (m+n)$ completion matrix $C_G(a,b)$.

A deformation satisfying Equation~(\ref{eqn:matrix_infinitesimally_rigid}) is called an infinitesimal deformation. A trivial deformation is one given by $\dot{a}_i = D^T a_i$ and $\dot{b}_j = -D b_j$ with $D \in \mathbb{R}^{r \times r}$, and every trivial deformation is infinitesimal. The framework $G(a,b)$ is called infinitesimally completable if all its infinitesimal motions are trivial. Since there are $r^2$ degrees of freedom choosing an invertible matrix $D$ and every trivial deformation is in the kernel of the completion matrix $C_G(a,b)$, then a non-trivial infinitesimal deformation exists if and only if the dimension of the kernel of the completion matrix $C_G(a,b)$ is equal to $r^2$. 

A framework $G(p)$ is locally completable if there exists a neighborhood $\mathcal{N}$ of the framework $G(p)$ such that $G(p)$ is the only framework in the neighborhood $\mathcal{N}$ up to trivial deformations with the same inner products. As in the rigidity theory of bar and joint frameworks, local completability of a generic framework is equivalent to infinitesimal completability, and hence local completability is a generic property. Therefore one can talk about local completability of a bipartite graph. For the low-rank matrix completion problem this implies that although one does not know the factor matrices $A$ and $B$, one can check with probability one  whether a partial matrix is locally completable by checking whether the partial matrix with the same underlying graph constructed from generic $A$ and $B$ is locally completable.

A framework is globally completable if it is the only framework up to trivial deformations giving the same inner products. Singer and Cucuringu also conjecture a sufficient condition for global completability using rank of stress matrices.

In Sections~\ref{sec:rigid} and~\ref{sec:isolated}, we will establish the connection between rigidity theory and nonnegative matrix factorizations. Although a framework is defined similarly to the low-rank matrix completion setting, the definition of infinitesimal rigidity is different because of the nonnegativity requirement of the factorization. Essentially, a linear algebra problem becomes a convex geometry problem: Instead of computing the span of a completion matrix one has to compute  the conic hull of a factorization matrix.

\subsection{Geometric characterization of nonnegative rank} \label{sec:geometric_characterization}

Nonnegative rank can be characterized geometrically via nested polyhedral cones. We describe two equivalent constructions from the literature for matrices of equal rank and nonnegative rank.

The first description is due to Cohen and Rothblum~\cite{cohen1993nonnegative}. It defines $P$ as the convex cone spanned by the columns of $M$ and $Q$ as the intersection of $\mathbb{R}^{m}_{\geq 0}$ and the column span of $M$.  Let $(A,B)$ be a rank-$r$ factorization of $M$, and let $\Delta$ be the simplicial cone spanned by the columns of $A$.  Since $A$ and $M$ have the same column span, the cones $P$, $\Delta$ and $Q$ all span the same dimension-$r$ subspace of $\B R^m$.  If $A$ is nonnegative, then $\Delta$ is contained in the positive orthant, so $\Delta \subseteq Q$.  If $B$ is nonnegative then each column of $M$ is a conic combination of columns of $A$ with coefficients given by columns of $B$, hence $P \subseteq \Delta$.  Conversely, one can construct a size-$r$ nonnegative factorization $(A,B)$ from a dimension-$r$ simplicial cone $\Delta$ that is nested between $P$ and $Q$ by taking the generating rays of $\Delta$ to be the columns of $A$. Therefore the matrix $M$ has nonnegative rank $r$ if and only if there exists a simplicial cone $\Delta$ such that $P \subseteq \Delta \subseteq Q$. Gillis and Glineur defined the restricted nonnegative rank of $M$ as the smallest number of rays of a cone that can be nested between $P$ and $Q$~\cite{gillis2012geometric}, which is an upper bound on the nonnegative rank in the case that the rank and nonnegative rank differ.

The work of Vavasis~\cite{vavasis2009on} presents a second description of the same nested cones up to a linear transformation. Fix a particular rank factorization $(A,B)$ of $M$ (not necessarily nonnegative). All rank factorizations of $M$ have the form $(AC,C^{-1}B)$ where $C \in \mathbb{R}^{r \times r}$ is an invertible matrix.  Let $P$ be the cone spanned by the columns of $B$; let $\Delta$ be the cone spanned by the columns of $C$; let $Q$ be the cone that is defined by $\{x \in \mathbb{R}^r: Ax \geq 0\}$.  The linear map $A$ sends these three polyhedral cones to their counterparts in the first construction.

Zeros in a nonnegative factorization correspond to incidence relations between the three cones, $P$, $\Delta$ and $Q$. In particular, a zero in $A$ means that a ray of $\Delta$ lies on a facet of $Q$. A zero in $B$ means that a ray of $P$ lies on a facet of $\Delta$.

One often considers nested polytopes instead of nested cones. One gets nested polytopes from nested cones by intersecting the cones with an affine plane, which is usually defined by setting the sum of the coordinates to 1.

Below we present a different geometric picture to help understand when a rank-$r$ matrix has nonnegative rank $r$ and specifically when it lies on the boundary of the semialgebraic set.  We will however at times refer to the nested polytopes $P \subseteq \Delta \subseteq Q$.

\subsection{Nonnegative rank boundaries}
\label{sec:boundaries}
Fixing $m$, $n$ and $r$, let $\B R^{m\times n}$ denote the set of real $m\times n$ matrices, and $\B R^{m \times n}_{\leq r}$ the subset with rank at most $r$.  The set $\B R^{m \times n}_{\leq r}$ is algebraic, meaning it is cut out by polynomial equations on the entries, namely by the $(r+1)\times (r+1)$ minors.  There is an algebraic map
 \[ \mu: \B R^{m\times r} \times \B R^{r \times n} \to \B R^{m \times n} \]
given by matrix multiplication, and $\B R^{m \times n}_{\leq r}$ is its image.  Let $\C M^{m \times n}_{\leq r}$ be the subset of $\B R^{m \times n}_{\leq r}$ consisting of the matrices that also have nonnegative rank at most $r$.  This set is the image of $\mu$ restricted to the $m\times r$ and $r\times n$ matrices with nonnegative entries.  Certain combinations of these inequalities when mapped forward produce the polynomial inequalities that describe $\C M^{m \times n}_{\leq r}$ as a subset of $\B R^{m \times n}_{\leq r}$ (see Proposition \ref{prop:int}).  A set such as $\C M^{m \times n}_{\leq r}$ that is described by a finite number of polynomial equations and inequalities is called a {\em semialgebraic set}.  Its relative boundary has a finite number of (algebraic) {\em boundary components}, each where one of the defining inequalities attains equality.  The boundary components are themselves irreducible semialgebraic sets, each of dimension one lower than $\C M^{m \times n}_{\leq r}$.  Some of the boundary components of $\C M^{m \times n}_{\leq r}$ are straight-forward: for a matrix $M$ to have a nonnegative rank, each of its entries must be greater than or equal to zero.  These inequalities define the {\em trivial boundary components} of $\C M^{m \times n}_{\leq r}$.

Some boundary components of $\C M^{m \times n}_{\leq r}$ consist of matrices that have infinitesimally rigid factorizations.  Such factorizations are locally unique, so they are important for understanding which matrices have unique nonnegative factorizations.  Using the ideas of rigidity theory, we show in Section~\ref{sec:rigid} that infinitesimally rigid factorizations are characterized by certain patterns of zero entries in the factors.  We give several necessary conditions on zero patterns that can result in infinitesimally rigid factorizations.  These results generalize the previously known full characterization of such zero patterns for $r=3$ \cite{kubjas2015fixed}.  All boundary components of $\C M^{m \times n}_{\leq 3}$ come from infinitesimally rigid factorizations, and there is only one zero pattern up to row and column permutation and transposition.  For higher rank, characterizing these zero patterns is more complicated.  In addition, we show in Sections~\ref{sec:local_rigidity} and \ref{sec:MSvS_example} that for $r\geq 4$ there are other kinds of boundary components with no analogue in the rank 3 case, and some of these components do not lead to locally unique factorizations.

We will show in Section~\ref{sec:rigidity_and_boundaries} that when a matrix $M$ lies in the relative interior of $\C M^{m \times n}_{\leq r}$, the set of rank-$r$ nonnegative factorizations has the full dimension, so it is not uniquely decomposable.  On the other hand, if $M$ is positive and lies on the relative boundary, then the nonnegativity constraints cut down the set of nonnegative factorizations to lower dimension.  On some types of boundary components, the set of factorizations of $M$ is cut down to a single point, meaning the factorization is locally unique.  Moreover if $M$ lies on no other boundary components, this factorization is globally unique.  Understanding the boundary components of $\C M^{m \times n}_{\leq r}$ then also provides an understanding of which matrices have unique nonnegative factorizations.  The equations and inequalities describing the boundary components of each type provide semialgebraic conditions that can be checked on a matrix of rank $r$ to determine if it has a unique nonnegative rank-$r$ factorization.

\section{Infinitesimally rigid factorizations}\label{sec:rigid}

In this section, we will establish a connection between rigidity theory and nonnegative matrix factorizations. The setup is similar to the low-rank matrix completion case, although there are three main differences: The graph $G$ is always a complete bipartite graph, there are additional nonnegativity constraints, and the space of ``trivial'' deformations is much smaller. We will assume that nonnegative rank of a matrix is equal to its rank.

Let $G=(V,E)$ be  the complete bipartite graph on $m+n$ vertices. As before, the vertices $V$ correspond to $a_1,\ldots,a_m,b_1,\ldots,b_n$ and the edges $E$ correspond to the entries of a matrix $M$.  We consider an infinitesimal motion of a framework parametrized by $t$. In addition to preserving the inner products $M_{ij} = a^T_i b_j$ for all $i \in [m],j \in [n]$, also $a_i$ and $b_j$ need to stay positive. Hence an infinitesimal motion has to satisfy
\begin{eqnarray}
&a^T_i \dot{b}_j + \dot{a}^T_i b_j = 0 \text{ for } (i,j) \in [m] \times [n], \label{eqn:nonnegative_rigidity1}\\ 
&a_i + t\dot{a}_i \geq 0 \text{ for } i \in [m] \text{ and } t \in [0,\epsilon), b_j+t\dot{b}_j \geq 0 \text{ for } j \in [n] \text{ and } t \in [0,\epsilon) \label{eqn:nonnegative_rigidity2}
\end{eqnarray}
for some $\epsilon >0$. 

As before, let $A$ and $B$ be the rank-$r$ matrices with rows $a_1^T,\ldots,a_m^T$ and columns $b_1,\ldots,b_n$ respectively. Similarly, define $\dot{A}$ and $\dot{B}$ to be the matrices with rows $\dot{a}_1^T,\ldots,\dot{a}_m^T$ and columns $\dot{b}_1,\ldots,\dot{b}_n$ respectively. Then $M = AB$ and Equation~(\ref{eqn:nonnegative_rigidity1}) can be expressed as $A\dot{B} + \dot{A}B = 0$.  For the equation to hold, the column span of $\dot{A}$ must be contained in that of $A$ and similarly for the row spans of $\dot{B}$ and $B$.  Therefore $\dot{A} = AD_1$ and $\dot{B} = -D_2B$ for $r\times r$ matrices $D_1$ and $D_2$.  Moreover $-AD_2B + AD_1B = 0$ and the fact that $A$ and $B$ are full rank implies that $D_1 = D_2$.  Therefore every solution to Equation~(\ref{eqn:nonnegative_rigidity1}) has the form $\dot{a}_i = D^Ta_i$ and $\dot{b}_j = -D b_j$ with $D \in \mathbb{R}^{r \times r}$.  Conversely it can be checked that any $a_1,\ldots,a_m,b_1,\ldots,b_n$ with derivatives of this form satisfy Equation~(\ref{eqn:nonnegative_rigidity1}). The set of matrices $D \in \mathbb{R}^{r \times r}$ that define infinitesimal motions is 
 \[ W_{(A,B)} := \{D \in \mathbb{R}^{r \times r} \mid \exists \, \epsilon > 0 \text{ such that } A + tAD \geq 0, B-tDB \geq 0 \text{ for } t \in [0,\epsilon) \}.\]

If matrix $D$ is diagonal then $\dot{a}_i = D^Ta_i$ and $\dot{b}_j = -D b_j$ always define an infinitesimal motion, and such a motion is called trivial.

\begin{definition}\label{def:rigid}
A framework is {\em infinitesimally rigid} if all its infinitesimal motions are trivial. 
\end{definition}

An infinitesimal motion does not necessarily correspond to any actual smooth path through $(A,B)$ in the space of nonnegative factorizations of $M$, but only to a tangent direction that does not violate nonnegativity.
Thus infinitesimal rigidity is not a necessary (and also not a sufficient) condition for the uniqueness of a nonnegative matrix factorization. However, every infinitesimally rigid nonnegative factorization is locally rigid (Proposition~\ref{prop:zero-dim}) and local rigidity is a necessary condition for the uniqueness of a nonnegative matrix factorization. In fact, when Kruskal rank of a certain matrix is maximal possible, then a locally rigid nonnegative factorization is infinitesimally rigid (Proposition~\ref{prop:generic_locally_rigid_is_infinitesimally_rigid}). These results allow us to state in Section~\ref{sec:isolated} so far the strongest necessary condition for the uniqueness of a nonnegative factorization.

\begin{example} \label{example:nnr3_infinitesimal_rigidity}
A rank-3 matrix $M$ with positive entries is on the boundary of $\C M_{3}^{m\times n}$ if and only if all nonnegative factorizations of $M$ are infinitesimally rigid. This follows from the analysis of Mond, Smith and van Straten in  \cite[Lemma 4.3]{mond2003stochastic}. One can show that a size-3 infinitesimally rigid nonnegative factorization has up to permuting rows of $A$, permuting columns of $B$, simultaneously permuting columns of $A$ and rows of $B$, and switching $A$ and $B^T$ the following form
\begin{equation}\label{eqn:zero_pattern_3}
\begin{pmatrix}
0 & \cdot & \cdot\\
\cdot & 0 & \cdot\\
\cdot & \cdot & 0\\
\cdot & \cdot & 0\\
\cdot & \cdots & \cdot\\
\vdots & \ddots & \vdots\\
\cdot & \cdots & \cdot
\end{pmatrix}
\begin{pmatrix}
0 & \cdot & \cdot & \cdot & \cdots & \cdot\\
\cdot & 0 & \cdot & \vdots & \ddots & \vdots\\
\cdot & \cdot & 0 & \cdot & \cdots & \cdot
\end{pmatrix}.
\end{equation}
\end{example}

We will now study the set  $W_{(A,B)}$ of matrices $D \in \mathbb{R}^{r \times r}$ that define infinitesimal motions. The inequality $a_i + tD^Ta_i \geq 0$ is trivially satisfied for $t \in [0,\epsilon)$ and some $\epsilon >0$ for all positive coordinates of $a_i$. Hence a row $a^T_i$ of $A$ defines an inequality on the $j$th column of $D$ if and only if the $j$th coordinate of $a_i$ is zero. The corresponding inequality is $d^T_j a_i \geq 0$ where $d_j$ denotes the $j$th column of $D$. For each $i = 1,\ldots,m$, let $S_i \subseteq \{1,\ldots,r\}$ be the set of entries of $a_i$ that are zero. Then $D \in W_{(A,B)}$  satisfies $d^T_j a_i \geq 0$ for all $j \in S_i$.  Equivalently $\inner{ a_i e^T_j, D} \geq 0$, where $\inner{\cdot,\cdot}$ denotes entry-wise inner product on $r\times r$ matrices.

On the other hand, a column $b_i$ of $B$ defines an inequality on the $j$th row of $-D$ if and only if the $j$th coordinate of $b_i$ is zero. This inequality is $-d'_j  b_i  \geq 0$ where $d'_j$ denotes the $j$th row of $D$. For each $i = 1,\ldots,n$, let $T_i \subseteq \{1,\ldots,r\}$ be the set of entries of $b_i$ that are zero. Then $D \in W_{(A,B)}$ satisfies $-d'_j b_i  \geq 0$ or equivalently $\inner{- e_j b^T_i , D} \geq 0$ for $j \in T_i$.

Hence $W_{(A,B)}$ is a polyhedral cone and we have described it in terms of its facet inequalities, but it will often be easier to work with its dual cone. For each $i = 1,\ldots,m$, define $\C A_i = \{a_i e^T_j \mid j \in S_i\}$, and for each $i = 1,\ldots,n$, define $\C B_i = \{- e_j b^T_i \mid j \in T_i\}$. Then
 \[ W_{(A,B)}^\vee = \cone(\C A_1 \cup \cdots \cup \C A_m \cup \C B_1 \cup \cdots \cup \C B_n). \]

\begin{proposition} \label{theorem:infinitesimally_rigid_factorizations}
A nonnegative factorization $(A,B)$ is infinitesimally rigid if and only if 
$ W_{(A,B)}^\vee$ is isomorphic to $\mathbb{R}^{r^2-r}$ (meaning $ W_{(A,B)}^\vee$ is an $(r^2-r)$-dimensional real vector space).
\end{proposition}

\begin{proof}

If the cone $W_{(A,B)}$ consists only of $r \times r$ diagonal matrices then the dual cone $W_{(A,B)}^\vee$ consists of all $r \times r$ matrices that are zero along the diagonal. This is a linear space of dimension $r^2-r$. Conversely, if $W_{(A,B)}$ contains other matrices, then its dimension is strictly larger than $r$. Hence the dimension of the largest subspace contained in $W_{(A,B)}^\vee$ is strictly less than $r^2-r$.
\end{proof}

Proposition~\ref{theorem:infinitesimally_rigid_factorizations} gives an algorithm for checking whether a nonnegative factorization is infinitesimally rigid. For example, open source tool \texttt{Normaliz}~\cite{Normaliz} allows one to compute the largest linear subspace contained in a cone given by its extremal rays. However, Proposition~\ref{theorem:infinitesimally_rigid_factorizations} does not give insight how to construct infinitesimally rigid nonnegative matrix factorizations. To solve this problem, we give a completely combinatorial necessary condition for a nonnegative matrix factorization to be infinitesimally rigid. In Appendix~\ref{sec:realizations_of_infinitesimally_rigid_factorizations}, we will use this result to construct  infinitesimally rigid nonnegative matrix factorizations for $5 \times 5$ matrices of nonnegative rank four, which is the first nontrivial case.

\begin{theorem}\label{prop:Wpoint}
 If $(A,B)$ is an infinitesimally rigid nonnegative rank-$r$ factorization then
 \begin{itemize}
  \item $A$ and $B$ have at least $r^2-r + 1$ zeros in total and
  \item for every distinct pair $i, j$ taken from $1,\ldots,r$, there must be a row of $A$ with a zero in position $i$ and not in position $j$.  Similarly for the columns of $B$.
 \end{itemize}
\end{theorem}
\begin{proof}
A nonnegative factorization $(A,B)$ being infinitesimally rigid is equivalent to $W_{(A,B)}^\vee \cong \B R^{r^2-r}$.  To express $\B R^{r^2-r}$ as the convex cone of a finite number of vectors requires at least $r^2-r+1$ vectors.  The size of the generating set defining $W_{(A,B)}^\vee$ is equal to the total number of zeros in $A$ and $B$.  
 
 The vectors coming from $A$ are nonnegative and the ones from $B$ are nonpositive.  If $W_{(A,B)}^\vee \cong \B R^{r^2-r}$, for each coordinate there must be at least one vector with a strictly positive value there, and one with a strictly negative value.  To get a positive value in coordinate $d_{ij}$ requires $A$ to have a row with zero in the $j$th entry and a non-zero value in the $i$th entry.  Similarly for columns of $B$.
\end{proof}

The second condition is a necessary condition for the uniqueness of nonnegative matrix factorization, see~\cite{laurberg2008theorems,gillis2012sparse}.

\begin{example}
Let $r=3$. The zero pattern~(\ref{eqn:zero_pattern_3}) is the unique zero pattern with seven zeros that fulfills the conditions in Theorem~\ref{prop:Wpoint}, up to allowed permutations.
\end{example}

We conclude this section with some properties of infinitesimally rigid nonnegative factorizations.

\begin{corollary} \label{corollary:infinitesimally_rigid_positive}
If $(A,B)$ is an infinitesimally rigid nonnegative rank-$r$ factorization with exactly $r^2 - r+1$ zeros, then $AB$ is strictly positive.
\end{corollary}

\begin{proof}
If $(A,B)$ is infinitesimally rigid, then the dual cone $W^{\vee}_{(A,B)}$ is equal to the space of matrices with zero diagonal of dimension $r^2 - r$.  The zeros of $A$ and $B$ correspond to the elements of a distinguished generating set of $W^{\vee}_{(A,B)}$ as described above.  A generating set of size $r^2-r+1$ is minimal, so the only linear relation among the generators must be among all $r^2-r+1$.

If $AB$ has a zero in entry $ij$ then row $a_i$ of $A$ and column $b_j$ of $B$ have zeros in complementary positions so that $a_i \cdot b_j = 0$.  Since the support of $b_j$ is contained in the set of columns for which $a_i$ is zero, the outer product matrix $a_i^Tb_j^T$ can be expressed as a nonnegative combination of the dual vectors coming from $a_i$.  Similarly, the matrix $-a_i^Tb_j^T$ can be expressed as a nonnegative combination of the dual vectors coming from $b_j$.  Summing these gives a linear relation among a strict subset of the generators, which is a contradiction.
\end{proof}

\begin{corollary}\label{lemma:number_of_zeros_in_a_row}
If $(A,B)$ is an infinitesimally rigid nonnegative factorization, then there is at least one zero in every column of $A$ and in every row of $B$.
\end{corollary}

\begin{proof}
It follows directly from Theorem~\ref{prop:Wpoint}.
\end{proof}

\begin{corollary}
If $M$ is strictly positive and $(A,B)$ is an infinitesimally rigid nonnegative rank-$r$ factorization of $M$, then there are at most $r-2$ zeros in every row of $A$ and in every column of $B$.
\end{corollary}

\begin{proof}
Since $M$ is positive, no row of $A$ or column of $B$ can contain only zeros.
If a row of $A$ contains $r-1$ zeros, then there has to be a row of $B$ that does not contain any zero, because otherwise $AB$ would have a zero entry. This contradicts Corollary~\ref{lemma:number_of_zeros_in_a_row}. 
\end{proof}

\begin{lemma}\label{lemma:max_number_entries}
If $(A,B)$ is an infinitesimally rigid nonnegative rank-$r$ factorization with $r^2-r+1$ zeros, then there are at most $r-1$ zeros in every column of $A$ and in every row of $B$.
\end{lemma}

\begin{proof}
As in the proof of Corollary \ref{corollary:infinitesimally_rigid_positive}, the only linear relation among the generators of $W^{\vee}_{(A,B)}$ must be among all $r^2-r+1$ generators.  If there were $r$ zeros in the same column of $A$, then there would be $r$ generators of $W^{\vee}_{(A,B)}$ contained in a $r-1$ dimensional subspace, implying a smaller linear relation which is impossible.  Similarly for the case of $r$ zeros in a row of $B$.
\end{proof}

This argument can be generalized to forbid other configurations of zeros that concentrate too many generators of $W^{\vee}_{(A,B)}$ into too small a support.
\begin{lemma}\label{lemma:zero rectangles}
 Let $(A,B)$ be an infinitesimally rigid nonnegative rank-$r$ factorization with $r^2-r+1$ zeros.  Let $\alpha,\beta \subseteq [r]$ and suppose $A$ has a $k \times |\alpha|$ submatrix of zeros with columns $\alpha$, and $B$ has a $|\beta| \times \ell$ submatrix of zeros with rows $\beta$.  Then
  \[ k|\alpha| + \ell|\beta| \leq (r-|\alpha|)|\alpha| + (r-|\beta|)|\beta| - |\alpha \setminus \beta||\beta \setminus \alpha|. \]
\end{lemma}

\begin{proof}
As in the proof of Corollary \ref{corollary:infinitesimally_rigid_positive}, a generating set of size $r^2-r+1$ is minimal, so the only linear relation among the generators must be among all $r^2-r+1$.  It can be checked that the zeros of $A$ described above correspond to $k|\alpha|$ generators of $W^{\vee}_{(A,B)}$ supported on entries $([r] \setminus \alpha) \times \alpha$.  Similarly the zeros of $B$ corresponds to $\ell|\beta|$ generators supported on entries $\beta \times ([r] \setminus \beta)$.  The intersection of these two supports is $(\beta \setminus \alpha) \times (\alpha \setminus \beta)$.  The number of generators cannot exceed the number of entries they are supported on, which gives the inequality.
\end{proof}

Lemma \ref{lemma:max_number_entries} is the special case when $\alpha$ is a singleton and $\beta$ is empty or the reverse, and this case seems to be the most applicable condition when $r$ is small.

\section{Locally rigid factorizations}\label{sec:isolated}

\subsection{Definition and properties}

\begin{definition}
A nonnegative factorization $(A,B)$ is {\em locally rigid} if all nonnegative factorizations of $AB$ in a neighborhood of $(A,B)$ are obtained by scaling the columns of $A$ and rows of $B$.
\end{definition}

If a matrix has a unique size-$r$ nonnegative factorization, then this factorization has to be locally rigid. We recall that the second condition in Theorem~\ref{prop:Wpoint} is a necessary condition for the uniqueness of a nonnegative matrix factorization by~\cite[Theorem 3]{laurberg2008theorems}. In fact, it is a necessary condition for local rigidity of a nonnegative matrix factorization using the argument in~\cite[Remark 7]{gillis2012sparse}. 

It concludes from the definition of an infinitesimally rigid nonnegative factorization that all nonnegative factorizations in some neighborhood are obtained from scalings. 
\begin{proposition}\label{prop:zero-dim}
 If nonnegative factorization $(A,B)$ is infinitesimally rigid, then it is locally rigid.
\end{proposition}

We will see in the next subsection that the converse is true if a certain matrix achieves its maximal possible Kruskal rank.

\begin{example} \label{example:nonnegative_rank_three_locally_rigid}
It follows from the discussion in Example~\ref{example:nnr3_infinitesimal_rigidity} that if $M$ lies on the boundary of $\C M_{3}^{m\times n}$, then all its nonnegative factorizations are locally rigid. This can be also seen using the geometric characterization of boundaries in~\cite[Corollary 4.4]{kubjas2015fixed}. Namely, a matrix with positive entries lies on the boundary of $\C M_{3}^{m \times n}$ if and only if for every nonnegative factorization of the matrix the corresponding geometric configuration satisfies that (i)~every vertex of the intermediate triangle lies on an edge of the outer polygon, (ii)~every edge of the intermediate triangle contains a vertex of the inner polygon, (iii)~a vertex of the intermediate triangle coincides with a vertex of the outer polygon or an edge of the intermediate triangle contains an edge of the inner polygon.  Such geometric configurations are isolated for fixed inner and outer polygons, hence the corresponding nonnegative factorizations are locally rigid.
\end{example}

In the rest of the subsection, we will explore modifications of locally rigid nonnegative matrix factorizations.

\begin{lemma}  \label{lemma:locally_rigid_remove_nonzero}
Let $(A,B)$ be a locally rigid factorization. Let $(A',B')$ be a factorization that is obtained from $(A,B)$ by erasing all rows of $A$ and columns of $B$ that do not contain any zero entries. Then $(A',B')$ is locally rigid. 
\end{lemma}

We will postpone proof of Lemma~\ref{lemma:locally_rigid_remove_nonzero} until Section~\ref{sec:geometry_of_factorizations} where we take a more geometric view on rigidity.

\begin{lemma}\label{lem:extend}
Let $(A,B)$ be a nonnegative factorization.  For $\varepsilon > 0$ small enough, there exists $A'$ obtained from $A$ by adding at most $r$ strictly positive rows and $B'$ obtained from $B$ by adding at most $r$ strictly positive columns such that any nonnegative factorization of $A'B'$ is in the $\varepsilon$-neighborhood of $(A'P,P^{-1}B')$ for some $r\times r$ scaled permutation matrix $P$.
\end{lemma}

\begin{proof}
Consider the geometric configuration of cones in $\mathbb{R}^r$ corresponding to the factorization $(A,B)$. Since $(A,B)$ is a nonnegative factorization, the intermediate cone is spanned by the unit vectors. We add $r$ strictly positive rows to $A$ that correspond to hyperplanes at most distance $\delta$ from the facets of the intermediate cone. We add $r$ strictly positive columns to $B$ that correspond to points that are at most distance $\delta$ from the vertices of the intermediate cone. Neither of these operations changes incidence relations between the three cones. The new outer cone is contained in $(1+\delta)$ times larger copy of the intermediate cone and the new inner cone contains a $(1-\delta)$ times smaller copy of the intermediate cone. For $\varepsilon$ small enough, there exists $\delta$ such that the only other cones with $r$ rays that one can be nested between a larger and smaller copy of the intermediate cone give factorizations that are in the $\varepsilon$-neighborhood of $(A'P,P^{-1}B')$.
\end{proof}

\begin{definition}
A nonnegative factorization $(A,B)$ is {\em globally rigid} if all nonnegative factorizations of $AB$  are obtained by scaling and permuting the columns of $A$ and rows of $B$.
\end{definition}

\begin{corollary} \label{cor:locally_rigid_to_globally_rigid}
Given a locally rigid nonnegative factorization $(A,B)$, then by adding at most $r$ strictly positive rows to $A$ and at most $r$ strictly positive columns to $B$, one can get a globally rigid nonnegative matrix factorization.
\end{corollary}

\subsection{When is infinitesimal rigidity equivalent to local rigidity?}\label{sec:local_rigidity}

Let $Z_{(A,B)}$ be a matrix with columns equal to the elements of $\C A_1 \cup \cdots \cup \C A_m \cup \C B_1 \cup \cdots \cup \C B_n$. Let $c$ be the number of columns of $Z_{(A,B)}$. Let the Kruskal rank be the maximal value $k$ such that any $k$ columns are linearly independent. We denote the Kruskal rank of $Z_{(A,B)}$ by $\Krank(Z_{(A,B)})$. We will show that if $\Krank(Z_{(A,B)}) = \min(c,r^2-r)$, then local rigidity implies infinitesimal rigidity. This result can be seen as an adaptation of Theorem~\ref{thm:Asimow_Roth} by Asimow and Roth to nonnegative matrix factorizations.

\begin{proposition} \label{prop:generic_locally_rigid_is_infinitesimally_rigid}
If $(A,B)$ is a nonnegative factorization that is locally rigid but not infinitesimally rigid, then $\Krank(Z_{(A,B)}) < \min(c,r^2-r)$.
\end{proposition}

\begin{proof}
We assume that $(A,B)$ is nonnegative factorization that is locally rigid but not infinitesimally rigid. We will show that $r < \dim W_{(A,B)} < r^2$.  The first inequality follows immediately from the fact that $(A,B)$ is not infinitesimally rigid.  The second inequality follows from the fact that $(A,B)$ is locally rigid by applying either Proposition \ref{prop:Wslice}, or the following argument that does not require the machinery of Section \ref{sec:rigidity_and_boundaries}.

Since $(A,B)$ is not infinitesimally rigid there exists $D \in W_{(A,B)}$ that is not diagonal. If $(\dot{a}_i)_j=(D^T a_i)_{j}$ is strictly positive for all $(i,j)$ such that $(a_i)_j$ is zero and $(\dot{b}_j)_i=(-D b_j)_{i}$ is strictly positive for all $(i,j)$ such that $(b_j)_i$ is zero, then the corresponding motion gives nonnegative factorizations for all $t \in [0,\epsilon)$ for some $\epsilon$ small enough. Hence a necessary condition for a locally rigid nonnegative factorization that is not infinitesimally rigid is that $(\dot{a}_i)_j=(D^T a_i)_{j}=0$ for some $(i,j)$ such that $(a_i)_j = 0$ or $(\dot{b}_j)_i=(-D b_j)_{i}=0$ for some $(i,j)$ such that $(b_j)_i = 0$. Moreover, there exists at least one pair $(i,j)$ such that for all $D \in W_{(A,B)}$ we have $(a_i)_j=(D^T a_i)_{j}=0$ or $(b_j)_i=(-D b_j)_{i}=0$, because otherwise one could take a conic combination of matrices $D$ with $(D^T a_i)_{j}=0$ and $(-D b_j)_{i}=0$  for different $(i,j)$ to get an element of $W_{(A,B)}$ with no $(D^T a_i)_{j}=0$ or $(-D b_j)_{i}=0$. 

Without loss of generality we assume that  $(a_i)_j=(D^T a_i)_{j}=0$ for all $D \in W_{(A,B)}$. Hence $a_i e^T_j$ and $-a_i e^T_j$ both belong to the dual cone $W^\vee_{(A,B)}$.  Since the dual cone has a non-trivial lineality space, $\dim W_{(A,B)} < r^2$.

From the fact that $r < \dim W_{(A,B)} < r^2$, it follows that the dual cone, $W^\vee_{(A,B)}$, has dimension-$k$ lineality space with $0 < k < r^2-r$.  A generating set of $W^\vee_{(A,B)}$ has a subset of size at least $k+1$ that generates the lineality space, and any $k+1$ of those generators are linearly dependent.  Therefore $Z_{(A,B)}$ has $k+1$ columns that are linearly dependent, so $\Krank(Z_{(A,B)}) \leq k < r^2 - r$.  Because $k+1 \leq c$, this also implies $\Krank(Z_{(A,B)}) < c$.
\end{proof}

\begin{corollary} \label{corollary:necessary_condition_for_global_rigidity}
If a nonnegative factorization $(A,B)$ is locally rigid, then $W^V_{(A,B)} \cong \mathbb{R}^{r^2-r}$ or  $\Krank(Z_{(A,B)}) < \min(c,r^2-r)$.
\end{corollary}

Since local rigidity is a necessary condition for global rigidity, the conditions in Corollary~\ref{corollary:necessary_condition_for_global_rigidity} are necessary for the uniqueness of a nonnegative factorization. We will also state Corollary~\ref{corollary:necessary_condition_for_global_rigidity_combinatorial} that is a simplified version of Corollary~\ref{corollary:necessary_condition_for_global_rigidity}. Corollary~\ref{corollary:necessary_condition_for_global_rigidity_combinatorial} directly strengthens  the necessary condition for uniqueness in~\cite[Theorem 3]{laurberg2008theorems} that states that the support of any column of $A$ cannot be contained in the support of any other column of $A$ and the support of any row of $B$ cannot be contained in the support of any other row of $B$.

\begin{corollary} \label{corollary:necessary_condition_for_global_rigidity_combinatorial}
If $(A,B)$ is a globally rigid nonnegative factorization, then the support of any column of $A$ cannot be contained in the support of any other column of $A$, the support of any row of $B$ cannot be contained in the support of any other row of $B$, and the matrices $A$ and $B$ have at least $r^2-r+1$ zeros in total or $\Krank(Z_{(A,B)}) < \min(c,r^2-r)$.
\end{corollary}

Separability based sufficient conditions for uniqueness, e.g. in~\cite{donoho2004does} and~\cite{laurberg2008theorems}, satisfy the additional condition that $A$ and $B$ have at least $r^2-r+1$ zeros in total, because the separability condition quaratees that one of the factors has at least $r^2-r$ zeros and there is at least one additional zero coming from the zero pattern in the other factor. It is unknown which of the two additional conditions is satisfied by sufficiently scattered based sufficient conditions, discussed in~\cite{fu2019nonnegative}. Our methods do not compare directly with methods that guarantee identifiability under further assumptions such as orthogonality of a factor, maximal sparseness, volume minimization or maximization of the polytope associated to one of the factors.

Corollary~\ref{corollary:necessary_condition_for_global_rigidity} together with the necessary condition for uniqueness from~\cite[Theorem 3]{laurberg2008theorems} gives Algorithm~\ref{algorithm:local_rigidity} for determining infinitesimal and local rigidity of a nonnegative matrix factorization.

\begin{algorithm}[h]
\caption{Local rigidity of a size-$r$ nonnegative matrix factorization $(A,B)$}\label{algorithm:local_rigidity}
\begin{algorithmic}[1]
\Procedure{LocalRigidityNMF}{$A,B,r$}
   \If{the support of any column of $A$ (resp. row of $B$) is contained in the support of any other column of $A$ (resp. row of $B$)}
   \State \textbf{return} $(A,B)$ is not locally rigid.
   \Else
   \State Construct the matrix $Z_{(A,B)}$. Let $c$ be the number of columns of $Z_{(A,B)}$.
   \If{the Kruskal-rank of $Z_{(A,B)}$ is equal to $\min(c,r^2-r)$}
       \State construct the polyhedral cone $W^V_{(A,B)}$ spanned by the columns of $Z_{(A,B)}$.
       \If{$W^V_{(A,B)}$ is isomorphic to $\R^{r^2-r}$}
          \State \textbf{return} $(A,B)$ is locally and infinitesimally rigid.
       \Else
          \State \textbf{return} $(A,B)$ is not locally rigid.
       \EndIf    
       \Else
       \State \textbf{return} $(A,B)$ is not infinitesimally rigid; local rigidity cannot be determined. 
   \EndIf
   \EndIf
\EndProcedure
\end{algorithmic}
\end{algorithm}

To test global rigidity of a size-$r$ nonnegative matrix factorization $(A,B)$, one can run a program that searches numerically for size-$r$ nonnegative matrix factorizations of the matrix $AB$. If $(A,B)$ is not globally rigid, then we do not expect the program to output precisely $(A,B)$ up to permutations and scalings. On the contrary, if the program outputs only $(A,B)$ up to permutations and scalings over multiple runs, then this provides evidence towards $(A,B)$ being globally rigid. This approach is further discussed in Appendix~\ref{sec:realizations_of_infinitesimally_rigid_factorizations}.

In the rest of the section, we present a locally rigid factorization which is not infinitesimally rigid. The example we present is a modification of an example by Shitov~\cite{shitov2019nonnegative} that he uses to show that nonnegative rank depends on the field. His example is a matrix of nonnegative rank five, we present a geometric configuration corresponding to a matrix of nonnegative rank four. Checking local rigidity involves  studying  signs of second derivatives in addition to the requirements on zeros and first derivatives.

\begin{example} \label{example:Shitov}
The outer polytope $Q=\text{conv}(\Omega_1,\Omega_2,A_i,B_i,C_i:1 \leq i \leq 3)$ is a modification of a simplex. Let $\varepsilon  = 1/20$. Three vertices of this simplex are replaced by small triangles conv$(A_i,B_i,C_i)$, where
\begin{small}
\begin{gather*}
A_1 = (0, 1/3 + \varepsilon, 1/3 - \varepsilon, 1/3),\quad
B_1 = (0, 1/3, 1/3 + \varepsilon, 1/3 - \varepsilon),\quad
C_1 = (0, 1/3 - \varepsilon, 1/3, 1/3 + \varepsilon),\\
A_2 = (1/3, 0, 1/3 + \varepsilon, 1/3 - \varepsilon),\quad
B_2 = (1/3 - \varepsilon, 0, 1/3, 1/3 + \varepsilon),\quad
C_2 = (1/3 + \varepsilon, 0, 1/3 - \varepsilon, 1/3),\\
A_3 = (1/3 - \varepsilon, 1/3, 0, 1/3 + \varepsilon),\quad
B_3 = (1/3 + \varepsilon, 1/3 - \varepsilon, 0, 1/3),\quad
C_3 = (1/3, 1/3 + \varepsilon, 0, 1/3 - \varepsilon).
\end{gather*}
\end{small} 
The last vertex of the simplex is replaced by a small edge conv$(\Omega_1,\Omega_2)$. The vertices $\Omega_1$ and $\Omega_2$ are points on the line 
$$
\frac{1}{(1 + (0.416827-1)t)}(1/3, 1/3 - 2t, 1/3 + t, 
  0.416827t)
$$
that are sufficiently close to and on the opposite sides of $(1/3,1/3,1/3,0)$. For example, one can take $t$ to be equal to $1/40$ and $-1/40$. Here $0.416827$ is an approximate number and we will explain later how to get the exact value.

The intermediate simplex $\Delta$ is conv$(\Omega,V_1,V_2,V_3)$, where
\begin{gather*}
V_1 = (0,1/3,1/3,1/3),\quad
V_2 = (1/3,0,1/3,1/3),\quad
V_3 = (1/3,1/3,0,1/3),\quad
\Omega = (1/3,1/3,1/3,0).
\end{gather*} 
The vertex $\Omega$ lies on the edge conv$(\Omega_1,\Omega_2)$ of the outer polytope. All other vertices $V_i$ lie on the triangles conv$(A_i,B_i,C_i)$.

The inner polytope $P$ is conv$(W,W_i,F_{ij},H:1 \leq i \leq 3, 1 \leq j \leq 2)$, 
where
\begin{small}
\begin{gather*}
W_1 = (1 - 3\varepsilon )V_1 + \varepsilon V_2 + \varepsilon V_3 + \varepsilon \Omega,\quad
W_2 = \varepsilon V_1 + (1 - 3\varepsilon )V_2 + \varepsilon V_3 + \varepsilon \Omega,\\
W_3 = \varepsilon V_1 + \varepsilon V_2 + (1 - 3\varepsilon )V_3 + \varepsilon \Omega,\quad
W = \varepsilon V_1 + \varepsilon V_2 + \varepsilon V_3 + (1 - 3\varepsilon )\Omega,\\
F_{11} = 0.81V_2 + 0.01V_3 + 0.18\Omega,\quad
F_{12} = 0.14V_2 + 0.20V_3 + 0.66\Omega,\\
F_{21} = 0.43V_1 + 0.22V_3 + 0.35\Omega,\quad
F_{22} = 0.20V_1 + 0.49V_3 + 0.31\Omega,\\
F_{31} = 0.11V_1 + 0.87V_2 + 0.02\Omega,\quad
F_{32} = 0.43V_1 + 0.12V_2 + 0.45\Omega,\\
H = 1/3V_1 + 1/3V_2 + 1/3V_3.
\end{gather*}
\end{small}
It has one vertex close to every vertex of the intermediate simplex: The vertex $W$ is  close to $\Omega$ and the vertices $W_i$ are close to $V_i$. Moreover, there are two vertices on each facet of the intermediate simplex besides the facet that is opposite to $\Omega$: The vertices $F_{ij}$ lie on the facet of the simplex spanned by all vertices but $V_i$. The interior polytope also contains the vertex $H$ that lies on the facet of the intermediate simplex that is opposite to $\Omega$. 

\begin{figure}
\centering
\begin{subfigure}[b]{0.3\textwidth}
\includegraphics[height=3cm]{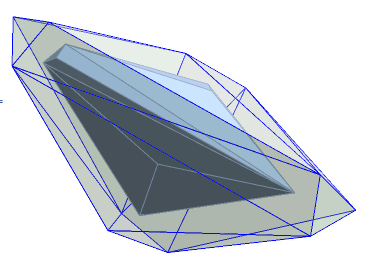}
\caption{$P \subseteq Q$}
\end{subfigure}
\begin{subfigure}[b]{0.3\textwidth}
\includegraphics[height=3cm]{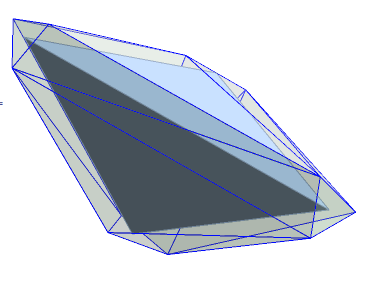}
\caption{$\Delta \subseteq Q$}
\end{subfigure}
\begin{subfigure}[b]{0.3\textwidth}
\includegraphics[height=3cm]{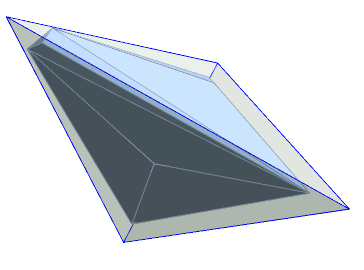}
\caption{$P \subseteq \Delta$}
\end{subfigure}
\caption{The pairwise inclusions of the three polytopes $P \subseteq \Delta \subseteq Q$ in Example~\ref{example:Shitov}.}
\label{figure:Shitov}
\end{figure}

The pairwise inclusions of the three polytopes are depicted in Figure~\ref{figure:Shitov}. The matrix $M$ corresponding to this geometric configuration is obtained by evaluating the facets of the outer polytope $Q$ at the vertices of the inner polytope $P$. The facets of $Q$ can be found for example using \texttt{polymake}~\cite{gawrilow1997polymake}. The matrix $A$ in the nonnegative factorization is obtained by evaluating the facets of $Q$ at the vertices of $B$; the matrix $B$ is obtained by evaluating the facets of $Q$ at the vertices of $P$. The nonnegative factorization has the following zero pattern (after removing rows of $A$ and columns of $B$ that do not contain zeros):

\begin{small}
$$
\begin{pmatrix}
0 & \cdot & \cdot & \cdot\\
\cdot & 0 & \cdot & \cdot\\
\cdot & \cdot & 0 & \cdot\\
\cdot & \cdot & \cdot & 0\\
\cdot & \cdot & \cdot & 0
\end{pmatrix}
\begin{pmatrix}
0 & 0 & \cdot & \cdot & \cdot & \cdot & \cdot \\
\cdot & \cdot & 0 & 0 & \cdot & \cdot & \cdot \\
\cdot & \cdot & \cdot & \cdot & 0 & 0 & \cdot\\
\cdot & \cdot & \cdot & \cdot & \cdot & \cdot & 0
\end{pmatrix}
$$
\end{small}
The number of zeros in this factorization is $12$, so this factorization is not infinitesimally rigid. We will show that it is locally rigid, i.e. that $\Delta$ is the only simplex that can be nested between $P$ and $Q$. The proof is analogous to the proof in~\cite{shitov2019nonnegative}. We present it here so that the reader is able to directly check the correctness of our example.

Since $P$ and $Q$ are constructed such that they are close to $\Delta$, any other simplex $\Delta'$ that can be nested between $P$ and $Q$ must be close to $\Delta$. We will give a parametrization of simplices and show that any simplex $\Delta'$ close to $\Delta$ can be parametrized in such a way.

We restrict to the affine plane in $\mathbb{R}^4$ defined by $x_1+x_2+x_3+x_4=1$. Let $\omega$ be a point on this plane such that $\|\Omega-\omega\|<\varepsilon$. Let $f_{ij} = (F_{ij}-(0,0,0,v_{ij}))/(1-v_{ij})$ where $v_{ij}$ are parameters. Let $H_i(\omega,v)$ be the hyperplane through $f_{i1},f_{i2},\omega$. Define the point $V_i(\omega,v)$ as the intersection of the hyperplanes $x_i=0$ and $H_j(\omega,v)$ for $j \neq i$. Define $\Delta(\omega,v) = \text{conv}(V_1(\omega,v),V_2(\omega,v),V_3(\omega,v),\omega)$. Then $\Delta = \Delta(\Omega,0)$. 

Since $\Delta'$ is close to $\Delta$, the facet of $\Delta'$ opposite to the vertex $V'_i$ intersects the line of $f_{i1}$'s and the line of $f_{i2}$'s. Moreover, the points where the facet of $\Delta'$ intersects these lines correspond to nonnegative $v_{ij}$, because $F_{ij} \in P \subset \Delta'$ correspond to zero parameters and going outwards from $P$ on the line of $f_{ij}$'s increases the value of the parameters $v_{ij}$. Furthermore, since maximal simplices inside $Q$ have vertices on the boundary of $Q$, we can assume that this is the case for $\Delta'$ and hence $\Delta'=\Delta(\omega,v)$ for some $\omega \in Q$ and $v \geq 0$.

Let $\Psi(\omega,v)=\det(V_1(\omega,v),V_2(\omega,v),V_3(\omega,v),H)$. We note that $\Psi(\Omega,0)=0$ and \linebreak $\det(V_1,V_2,V_3,\Omega)>0$. To show that $\Delta$ is the only simplex that can be nested between $P$ and $Q$ it is enough to show that for all other $\Delta(\omega,v)$ close to $\Delta$ with $\omega \in Q$ and $v \geq 0$, we have $H \not \in \Delta(\omega,v)$. This is equivalent to $\Psi(\omega,v)<0$ and $(\Omega,0)$ being a local maximum of $\Psi$ when $\omega \in Q$ and $v \geq 0$. It can be checked that the partial derivatives $\partial \Psi / \partial v_{ij}$ and the directional derivatives in the directions from $(\Omega,0)$ to $(A_i,0),(B_i,0),(C_i,0)$ are negative at $(\Omega,0)$. Finally, on the line 
$$
\frac{1}{(1 + (0.416827-1)t)}(1/3, 1/3 - 2t, 1/3 + t, 
  0.416827t),
$$
we have $\Psi'=0$ and $\Psi''<0$. In fact, the number $0.416827$ is  an approximation of the solution for $x$ in the equation $\frac{\partial \Psi((1/3,1/3-2t,1/3+t,xt),v)}{\partial t}|_{(t=0,v=0)}=0$.

This example is a modification of an infinitesimally rigid example with $13$ zeros where a vertex of the outer polytope is replaced with an edge $\text{conv}(\Omega_1,\Omega_2)$. The corresponding nonnegative factorization would have an extra row in $A$ with zero in the last column. The vertex of the intermediate simplex that for the infinitesimally rigid configuration coincides with the vertex of the outer polytope lies now on the new edge. The only difference between the two examples is that theoretically one can now move the vertex of the intermediate simplex also along the edge $\text{conv}(\Omega_1,\Omega_2)$, but in fact this is not possible, because the local maximum of $\Psi$ on $\text{conv}(\Omega_1,\Omega_2)$ is $\Omega$. By results of Mond, Smith and van Straten~\cite{mond2003stochastic}, it is not possible to construct an analogous example for polygons.
\end{example}

\section{Rigidity and boundaries} \label{sec:rigidity_and_boundaries}
In this section we use $\C M^{m\times n}_{r}$ to denote the set of $m\times n$ matrices with rank and nonnegative rank both equal to exactly $r$.
A matrix of nonnegative rank three is on the boundary of $\C M_{3}^{m\times n}$ if and only if it has a zero entry or all its nonnegative factorizations are infinitesimally rigid. The goal of this section is to study the connection between boundaries of $\C M_{r}^{m\times n}$ and rigidity theory for $r \geq 4$. We already saw in Example~\ref{example:Shitov} that there exist locally rigid nonnegative matrix factorizations that are not infinitesimally rigid. Combining this with results in Section~\ref{sec:geometry_of_factorizations}, one can show that there exists a matrix on the boundary of $\C M_{4}^{m\times n}$ for $m,n$ large enough that has a locally rigid nonnegative factorization but no infinitesimally rigid nonnegative factorizations. Furthermore, in Section~\ref{sec:MSvS_example} we will show that there exist strictly positive matrices on the boundary of $\C M_{r}^{m\times n}$ for $r \geq 4$ that have nonnegative factorizations that are not locally rigid. There exists a neighborhood of a such factorization whose dimension is strictly between $r$ and $r^2$, the minimal and maximal dimensions of spaces of factorizations.

\subsection{Geometry of nonnegative matrix factorizations} \label{sec:geometry_of_factorizations}

As in Section~\ref{sec:boundaries} let $\mu$ be the usual matrix multiplication map, but we now restrict the domain to pairs of matrices with full rank $r$:
 \[ \mu: \B R_r^{m\times r} \times \B R_r^{r\times n} \to \B R_r^{m \times n}. \]
The image of $\mu$ is $\B R_r^{m\times n}$, the set of $m\times n$ matrices with rank $r$.  The positive orthant $(\B R_r^{m\times r})_{\geq 0} \times (\B R_r^{r\times n})_{\geq 0}$ is mapped onto $\C M_{r}^{m \times n}$, the set of rank-$r$ matrices with nonnegative rank $r$.  The {\em trivial boundary} of $\C M_{r}^{m \times n}$ consists of such matrices with at least one zero entry.

Fix a rank-$r$ matrix $M$ with strictly positive entries and a rank factorization $(A,B)$ with $M = AB$.  The set of all rank factorizations of $M$ is the fiber
 \[ \mu^{-1}(M) = \{(AC,C^{-1}B)\mid C \in \B R^{r\times r} \text{ invertible} \}. \]
This set is a real $r^2$-dimensional smooth irreducible variety.  Let
 \[ F := \{(C,C^{-1})\mid C \in \B R^{r\times r} \text{ invertible} \} \subseteq \B R^{r\times r} \times \B R^{r\times r} . \]
$F$ is the graph of the inverse function on $r\times r$ matrices.  The injective linear map
\[ \nu_{(A,B)}: \B R^{r\times r} \times \B R^{r\times r} \to \B R^{m\times r} \times \B R^{r\times n} \]
\[ (C,D) \mapsto (AC, DB) \]
sends $F$ to $\mu^{-1}(M)$.  The image of $\nu_{(A,B)}$ is the subspace of pairs $(\alpha,\beta)$ such that the columns of $\alpha$ are in the columns span of $A$ and the rows of $\beta$ are in the row span of $B$.
 
 \begin{proposition}
 The map $\mu: \B R_r^{m\times r} \times \B R_r^{r\times n} \to \B R_r^{m \times n}$ is a fiber bundle, with fiber $F$.
\end{proposition}
\begin{proof}
 A matrix $M \in \B R^{m \times n}_r$ has a set of $r$ linearly independent columns.  Given a rank factorization $(A,B)$ of $M$, the same set of columns is linearly independent in $B$.  Call the $r \times r$ submatrix they form $C$.  Then $(AC, C^{-1}B)$ is a rank factorization of $M$ with $C^{-1}B$ having the $r \times r$ submatrix in these columns equal to the identity, and this is the unique factorization of $M$ with that property.  Let $K$ be the subset of $\B R^{m\times r}_r \times \B R^{r\times n}_r$ of pairs $(\alpha, \beta)$ in which $\beta$ has this particular submatrix equal to the identity.
 
 All matrices in $\B R^{m \times n}_r$ have a unique factorization in $K$ unless there is linear dependence among the chosen columns.  Such exceptions form a lower dimensional subset, so in particular $M$ has a neighborhood $X$ of matrices with factorizations in $K$.  Then $\mu^{-1}(X)$ has product structure $(\mu^{-1}(X)\cap K) \times F$ by map
  \[ ((\alpha, \beta), (\gamma,\gamma^{-1})) \mapsto (\alpha \gamma, \gamma^{-1} \beta) \]
 which can be checked is continuous with continuous inverse.  This proves the fiber bundle structure of $\mu$.
\end{proof}

A factorization $(AC,C^{-1}B)$ of $M$ is a nonnegative factorization of $M$ if $AC \in (\B R_r^{m\times r})_{\geq 0}$ and $C^{-1}B \in (\B R_r^{r\times n})_{\geq 0}$.  Let $c_1,\ldots,c_r$ denote the columns of $C$ and $c'_1,\ldots,c'_r$ the rows of $C^{-1}$.  The inequality $Ac_i \geq 0$ gives $m$ linear inequalities on $c_i$ and defines a polyhedral cone in $\B R^r$ with at most $m$ facets which we will denote $P_A$.  Similarly $c'_iB \geq 0$ defines a polyhedral cone $P_{B^T}$ in $(\B R^r)^*$ with at most $n$ facets.  The nonnegative factorizations of $M$ then correspond to the set $F \cap (P_A^{\times r} \times P_{B^T}^{\times r})$.  Let $U_{(A,B)} := (P_A^{\times r} \times P_{B^T}^{\times r})$, which is itself a polyhedral cone.

Fixing $M$ and a rank factorization $(A,B)$, the injective linear map $\nu_{(A,B)}$ that sends $F$ to $\mu^{-1}(M)$ also maps cone $U_{(A,B)}$ to $(\B R_r^{m\times r})_{\geq 0} \times (\B R_r^{r\times n})_{\geq 0} \cap \im(\nu_{(A,B)})$.  The boundary of $U_{(A,B)}$ maps to pairs of matrices that have at least one zero entry.  Because $M$ is assumed to have positive entries, $\im(\nu_{(A,B)})$ is not contained in a coordinate hyperplane of $\B R_r^{m\times r}) \times \B R_r^{r\times n})$.  Therefore the interior of $U_{(A,B)}$ maps to pairs of matrices with positive entries.  Sometimes it will be convenient to work in one or the other system of coordinates.

\begin{remark}
$P_A$ is the outer cone, $Q$, and $P_{B^T}$ is dual to the inner cone, $P$, in the second geometric characterization in Section~\ref{sec:geometric_characterization}.
\end{remark}

If $(A,B)$ is a nonnegative factorization of $M$ then $(AD,D^{-1}B)$ is as well for any diagonal matrix $D$ with positive diagonal entries.  We will generally be interested only in factorizations modulo this scaling.  

Now we have introduced the tools for proving Lemma~\ref{lemma:locally_rigid_remove_nonzero}.

\begin{proof}[Proof of Lemma~\ref{lemma:locally_rigid_remove_nonzero}]
Let $(A,B)$ be a locally rigid factorization. Let $(A',B')$ be a factorization that is obtained from $(A,B)$ by erasing all rows of $A$ and columns of $B$ that do not contain any zero entries.

For the sake of contradiction, assume that $(A',B')$ is not locally rigid. 
Equivalently every neighborhood of $(I,I)$ in $F\cap U_{(A',B')}$ contains a pair $(C,C^{-1})$ where $C$ is not diagonal. This implies that there is a row $a_i$ of $A$ with positive entries and a column $c_j$ of $C$ such that $a_i c_j <0$  or there is a column $b_i$ of $B$ with positive entries and a row $c'_j$ of $C^{-1}$ such that $c'_j b_i <0$. Let $c_{\max}$ be the maximal entry of $A$ and $B$; let $c_{\min}$ be the minimal non-zero entry of $A$ and $B$. Consider the $\varepsilon$-neighborhood of $(I,I)$ where $\varepsilon=\frac{c_{\min}}{c_{\min}+(r-1)c_{\max}}$. For any $(C,C^{-1})$ in this neighborhood, every non-diagonal entry of $C$ is greater than $-\varepsilon$ and every diagonal entry is greater than $1-\varepsilon$. Since $A$ and $B$ are nonnegative, we have $a_i c_j \geq -(r-1)\varepsilon c_{\max} + (1-\varepsilon) c_{min}=0$  and similarly $c'_j b_i \geq 0$ for all $i,j$.
\end{proof}

\begin{proposition}\label{prop:int}
Positive $M \in \C M_{r}^{m \times n}$ lies on boundary of $\C M_{r}^{m \times n}$ if and only if every nonnegative factorization $(A,B)$ of $M$ has at least one zero entry.
\end{proposition}
\begin{proof}
 Suppose $M$ has a strictly positive rank factorization $(A,B)$.  Then $(A,B)$ has a relatively open neighborhood $W$ contained in $\mu^{-1}(M) \cap (\B R_r^{m\times r})_{> 0} \times (\B R_r^{r\times n})_{> 0}$. Since $\mu$ is a fiber bundle, it is an open mapping.  Therefore $\mu(W)$ is an open neighborhood of $M$ in $\C M_{r}^{m \times n}$, so $M$ is in the interior.
 
 Suppose $M$ does not have any strictly positive rank factorizations.  Equivalently $F$ does not intersect the interior of $U_{(A,B)}$.  We will construct a rank-$r$ matrix $M'$ arbitrarily close to $M$ with $\nrank(M') > r$.  For cone $P_A \subseteq \B R^r$, let $P_A^\vee \subseteq (\B R^r)^*$ denote the dual cone, which consists of all linear functionals that are nonnegative on $P_A$, and similarly let $P_{B^T}^\vee$ be the dual cone of $P_{B^T}$.  Neither the cone $P_A$ nor $P_{B^T}$ contains a line since after a change of coordinates each are a subspace intersected with a positive orthant.  Therefore we can choose functionals $x$ and $y$ in the interiors of $P_A^\vee$ and $P_{B^T}^\vee$ respectively.  The functional $x$ has the property that for any non-zero $v \in P_A$, $xv > 0$, and similarly for $y$ with respect to $P_{B^T}$.
 
 Let $X$ be the $m \times r$ matrix with $x$ in every row, and $Y$ the $r\times n$ matrix with $y$ in every column.  Choose vectors $v$ and $w$ in the interiors of $P_A$ and $P_{B^T}$ respectively.  Let $A' = A - \epsilon X$ and $B' = B - \epsilon Y$ for $\epsilon > 0$ chosen small enough so that $v$ and $w$ are still in the interiors of $P_{A'}$ and $P_{(B')^T}$ respectively.  Then $U_{(A',B')}$ contains the point given by $r$ copies of $v$ and $r$ copies of $w$ that is in $U_{(A,B)}$.
 
 Let $(C,D)$ be any non-zero point on the boundary of $U_{(A,B)}$, so either $a_ic_j = 0$ for some row $a_i$ of $A$ and column $c_j$ of $C$ or $d_ib_j = 0$ for some row $d_i$ of $D$ and column $b_j$ of $B$.  Without loss of generality assume the first case.  Letting $a'_i$ denote the $i$th row of $A'$ we have $a'_ic_j = a_ic_j - \epsilon xc_j < 0$ because $c_j \in P_A$.  This implies $(C,D)$ is outside of the cone $U_{(A',B')}$.  Since $U_{(A',B')}\setminus\{0\}$ intersects the interior of $U_{(A,B)}$ but not its boundary, it must be contained in the interior of $U_{(A,B)}$.  Since $F$ does not intersect the interior of $U_{(A,B)}$ or the origin, it does not intersect $U_{(A',B')}$.  Therefore $M' = A'B'$ has $\nrank(M') > r$.
 
 Note that $M' = M - \epsilon(XB + AY) + \epsilon^2(XY)$, which can be made arbitrarily close to $M$ in 2-norm by choosing $\epsilon$ small enough.  For sufficiently small $\epsilon$, $A'$ and $B'$ have full rank since this is an open condition, so $\rank(M') = r$.
\end{proof}

\begin{proposition}\label{prop:interior}
 Positive $M$ has a strictly positive rank factorization if and only if the set of nonnegative rank factorizations of $M$ contains a nonempty subset that is open in the Euclidean subspace topology on $\mu^{-1}(M)$ (or equivalently the Zariski closure of $\mu^{-1}(M) \cap (\B R_r^{m\times r})_{\geq 0} \times (\B R_r^{r\times n})_{\geq 0}$ is $\mu^{-1}(M)$).
\end{proposition}
\begin{proof}
 First we show that the set $F$ is not contained in any facet hyperplane of $U_{(A,B)}$.  Every facet $H$ of $U_{(A,B)}$ is defined by a linear equation involving either only the first set of coordinates or only the second set.  Consider the former case without loss of generality.  Recall that $F$ is the graph of the inverse function on $r\times r$ matrices, so the first set of coordinates are algebraically independent in $F$.  Therefore $H\cap F$ has strictly lower dimension than $F$.
 
 Suppose an open neighborhood of $F$ is contained in $U_{(A,B)}$.  If the neighborhood is contained in the boundary of $U_{(A,B)}$ then $F$ is contained in the hyperplane of one of the facets since $F$ is irreducible.  As shown above, this cannot happen so there must be a point on $F$ in the interior of $U_{(A,B)}$.  Conversely, if $F \cap \interior(U_{(A,B)})$ is non-empty, it is open in the subspace topology on $F$ since $\interior(U_{(A,B)})$ is open.
\end{proof}

Suppose $\nrank(M) = r$, and that $(A,B)$ is a nonnegative factorization.
The point $(I,I) \in F$ has $\nu_{(A,B)}(I,I) = (A,B)$.  To understand the possible boundary components of sets of matrices with rank and nonnegative rank equal to $r$, it is sufficient to understand the ways that $F$ and $U_{(A,B)}$ can intersect in a neighborhood of $(I,I)$.  It is not true that if $F$ and $\interior(U_{(A,B)})$ are disjoint in a neighborhood of $(I,I)$, then $M$ is on the boundary of $\C M_{r}^{m \times n}$; they may intersect elsewhere.  However,
the following corollary to Lemma~\ref{lem:extend}  demonstrates we can always construct $M' = A'B'$ that has $M$ as a submatrix, is on the boundary, and for which $U_{(A',B')}$ agrees with $U_{(A,B)}$ in a neighborhood of $(I,I)$.

\begin{corollary}
Suppose positive matrix $M$ has a nonnegative factorization $(A,B)$ such that all nonnegative factorizations of $M$ in a neighborhood of $(A,B)$ have at least one zero entry.  Then there is a matrix $A' \in \B R_r^{m'\times r}$ obtained by adding at most $r$ strictly positive rows $A$ and a matrix $B' \in \B R_r^{r\times n'}$ obtained by adding at most $r$ strictly positive columns to $B$, such that $M' = A'B'$ is on the non-trivial boundary of $\C M_{r}^{m' \times n'}$.
\end{corollary}

We now consider the tangent space of $F$ at $(I,I)$, and how it intersects $U_{(A,B)}$. 
The tangent space of $F$ at $(I,I)$ is
 \[ T_{(I,I)} F = \{(D,-D)\mid D \in \B R^{r\times r}\}. \]
The cone $W_{(A,B)}$ from Section~\ref{sec:rigid} is the projection to the first $\B R^{r^2}$ factor of tangent directions $(D,-D)$ such that the line $(I+tD,I-tD)$ stays in $U_{(A,B)}$ for $t \in [0,\epsilon)$ for some $\epsilon > 0$. 
The tangent directions along the diagonal matrices $D$ always lie in $W_{(A,B)}$. We recall that a nonnegative factorization $(A,B)$ is infinitesimally rigid if $W_{(A,B)}$ consists only of the diagonal matrices, and it is  locally rigid if a neighborhood of $(I,I)$ in $F \cap U_{(A,B)}$ has dimension~$r$, the minimal possible dimension.

\begin{proposition}\label{prop:Wslice}
 If $W_{(A,B)}$ has full dimension $r^2$, then $M$ is in the interior of $\C M_{r}^{m \times n}$.
\end{proposition}
\begin{proof}
 As in the proof of Proposition \ref{prop:interior}, if $W_{(A,B)}$ has full dimension, then the tangent space $T_{(I,I)} F$ intersects the interior of $U_{(A,B)}$ in a neighborhood of $(I,I)$.  This implies that $F$ itself intersects the interior of $U_{(A,B)}$.  By Propositon \ref{prop:int} and Proposition \ref{prop:interior}, $M$ is in the interior of $\C M_{r}^{m \times n}$.
\end{proof}

In Example~\ref{example:Shitov}, a neighborhood of $(I,I)$ in $F \cap U_{(A,B)}$  has dimension $r$, but $\dim W_{(A,B)}>r$. In general, if $r < \dim W_{(A,B)} < r^2$ then this value may differ from the dimension of a neighborhood of $(I,I)$ in $F \cap U_{(A,B)}$ in either direction.

\begin{example}\label{ex:triangles}
 Consider the following rank 3 matrix with nonnegative rank 3 factorization
 \[ M = \begin{pmatrix}
         2&1&1\\
         1&2&1\\
         1&1&2
        \end{pmatrix}
 = \begin{pmatrix}
          0&1&1\\
         1& 0&1\\
         1&1& 0
        \end{pmatrix}
    \begin{pmatrix}
          0&1&1\\
         1& 0&1\\
         1&1& 0
        \end{pmatrix}. \]
 Here $W_{(A,B)}^\vee$ is the conic combination of the 6 vectors
 \[ \begin{pmatrix}0&0&0\\1&0&0\\1&0&0\end{pmatrix},
    \begin{pmatrix}0&1&0\\0&0&0\\0&1&0\end{pmatrix},
    \begin{pmatrix}0&0&1\\0&0&1\\0&0&0\end{pmatrix},
    \begin{pmatrix}0&-1&-1\\0&0&0\\0&0&0\end{pmatrix},
    \begin{pmatrix}0&0&0\\-1&0&-1\\0&0&0\end{pmatrix},
    \begin{pmatrix}0&0&0\\0&0&0\\-1&-1&0\end{pmatrix}\]
 corresponding to the 6 zeros in $A$ and $B$.
 This forms a 5 dimensional subspace of $\B R^9$ and $W_{(A,B)}$ is the orthogonal complement which is a space of dimension 4 (the 3 trivial diagonal directions plus 1),
  \[ W_{(A,B)} = \left\{ \begin{pmatrix}d_1&-t&t\\t&d_2&-t\\-t&t&d_3\end{pmatrix} \; \Bigg|\; t,d_1,d_2,d_3 \in \B R\right\}. \]
However any neighborhood of $(I,I)$ in $F \cap U_{(A,B)}$ has full dimension 9.  In fact $M$ is not on the algebraic boundary of $\C M_{3}^{3\times 3}$.  The geometry of the nested polytopes of this example are shown in Figure \ref{fig:triangles}.
\end{example}
\begin{figure}
\centering
 \includegraphics[width=0.7\textwidth]{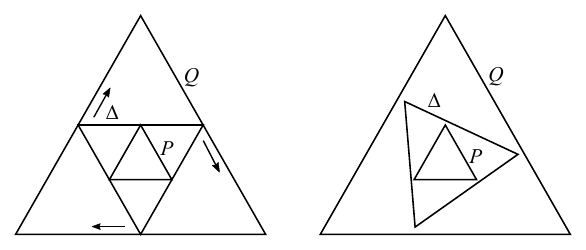}
 \caption{On the left are the nested polytopes $P \subseteq \Delta \subseteq Q$ corresponding to factorization $(A,B)$ of $M$ from Example \ref{ex:triangles}.  The arrows indicate a tangent direction in $W_{(A,B)}$.  On the right is a nearby factorization $(AC,C^{-1}B)$ with $(C,C^{-1})$ in the interior of $U_{(A,B)}$.}
 \label{fig:triangles}
\end{figure}

Suppose that factorization $(A,B)$ is not locally rigid, so $F \cap U_{(A,B)}$ has dimension larger than $r$ in a neighborhood of $(I,I)$.  If we suppose also that $(A,B)$ is a boundary factorization, then locally $F \cap U_{(A,B)}$ cannot exceed the cone $W_{(A,B)}$, which represents the local intersection of the tangent space $T_{(I,I)}$ and $U_{(A,B)}$ (in contrast to Example \ref{ex:triangles}).  Within this situation, there are two broad cases to consider: either $F \cap U_{(A,B)}$ is equal to $W_{(A,B)}$ in a neighborhood of $(I,I)$, or it is strictly contained in $W_{(A,B)}$. We will study the first case in Section \ref{sec:MSvS_example}. An example of the second case is the locally rigid nonnegative factorization that is not infinitesimally rigid in Example~\ref{example:Shitov}.

\subsection{Partially infinitesimally rigid factorizations} \label{sec:MSvS_example}
Here we present a construction to produce matrices of rank $r > 3$ that are on the non-trivial boundary of nonnegative rank $r$, and have a positive dimensional set of nonnegative factorizations.
\begin{definition}
 A nonnegative factorization $(A,B)$ is {\em partially infinitesimally rigid} if $W_{(A,B)}$ is equal to $F \cap U_{(A,B)}$ in a neighborhood of $(I,I)$, and $\dim W_{(A,B)} < r^2$.
\end{definition}
Partially infinitesimally rigid factorizations generalize infinitesimally rigid factorizations. When $\dim W_{(A,B)}$ exceeds $r$, the factorization $(A,B)$ is not rigid.  In the examples we have encountered, the nonnegative factorizations in a neighborhood of $(A,B)$ have some columns of $A$ fixed, while others have freedom.

For $F$ to contain the cone $W_{(A,B)}$, it must contain its affine hull, so we first examine the question: what affine linear spaces passing through $(I,I)$ are contained in $F$?  A line through $(I,I)$ has the form
 \[ (I + tD, I + tE). \]
To be contained in $F$, it must be that $(I+tD)(I+tE) = I$.  This holds exactly when $E = -D$ and $D^2 = 0$.  Therefore an affine linear space in $F$ through $(I,I)$ has the form
 \[ \{(I + D, I-D) \mid D \in V \} \]
where $V$ is some linear space of $r\times r$ matrices $D$ satisfying $D^2 = 0$.

One way to produce such a space $V$ is to choose a subspace $S \subseteq \B R^r$ and define
 \[ V_S = \{D \in \B R^{r\times r} \mid \im D \subseteq S \subseteq \ker D \}. \]
However not all spaces $V$ have this form, as the following example shows.  We do not know a full characterization of such spaces $V$.

\begin{example}
 Let $V$ be the space
\[ V = \left\{\begin{pmatrix} 0&s&t&0 \\
          0&0&0&t \\
          0&0&0&-s \\
          0&0&0&0
         \end{pmatrix} \; \Bigg|\; s,t \in \B R \right\}. \]
Each matrix $D \in V$ has $\im D = \ker D = \inner{e_1, te_2 - se_3}$, so there is no uniform space $S \subseteq \B R^4$ such that $\im D \subseteq S \subseteq \ker D$ for all $D \in V$.
\end{example}

We focus on the case of a space $V_S$ with $S$ a coordinate subspace of $\B R^r$ because we have a simple procedure to create factorizations $(A,B)$ for which $W_{(A,B)}$ has this form.

\begin{proposition}\label{prop:pir}
 Let $(A,B)$ be an infinitesimally rigid nonnegative rank-$r$ factorization.  There is a partially infinitesimally rigid nonnegative rank $r+1$ factorization $(A',B')$ where $A'$ is a $n \times (r+1)$ matrix obtained from $A$ by adding a positive column and $B'$ is a $(r+1) \times (m+1)$ matrix obtained from $B$ by adding a row of zeros and then a positive column.
\end{proposition}

\begin{proof}
 Let $S$ be $\inner{e_{1},\ldots,e_{r}}$.  Then $V_S$ consists of matrices that are supported only in the first $r$ entries of the last column.  We will construct the positive column added to $A$ such that
  \[ W_{(A',B')} = \inner{e_1 e_1^T,\ldots,e_{r+1} e_{r+1}^T} + V_S, \]
 This is equivalent to showing that $W_{(A',B')}^\vee$ is equal to the space of $(r+1)\times (r+1)$ matrices supported on the off-diagonal entries of the first $r$ columns.
 
 The positive column added to $B'$ is only to bring $B'$ up to full rank, $r+1$.  It does not contribute to $W_{(A',B')}^\vee$ and will not arise again in the proof.
 
 First we show that the linear span of $W_{(A',B')}^\vee$ is equal this space of matrices.  We characterize the generating set of $W_{(A',B')}^\vee$ coming from the zeros of $A'$ and $B'$.
 The natural embedding of each generator of $W_{(A,B)}^\vee$ of the form $-e_ib_j^T$ is a generator of $W_{(A',B')}^\vee$ since the columns of $B'$ are the columns of $B$ with a zero entry added to the end.  Each generator of the form $a_j^Te_i^T$ corresponds to $a_j^Te_i^T + a_{j,r+1}e_{r+1}e_i^T$ in $W_{(A',B')}^\vee$.  In addition, $W_{(A',B')}^\vee$ has generator $-e_{r+1}b_j^T$ for each $j = 1,\ldots,m$ coming from the new zero row added to $B'$.  It follows that $W_{(A',B')}^\vee$ is contained in the space claimed.  The generators of the form $-e_{r+1}b_j^T$ span $V_S^T$ since $B$ has full rank $r$.  Under the natural projection $\B R^{(r+1)\times (r+1)} \to \B R^{r\times r}$, the generating set of $W_{(A',B')}^\vee$ maps to the generating set of $W_{(A,B)}^\vee$ and zero, which span the $r\times r$ matrices with zero diagonal.  Therefore $W_{(A',B')}^\vee$ spans the matrices supported on the off-diagonal entries of the first $r$ columns.
 
 To prove that $W_{(A',B')}^\vee$ is a linear space, we show that zero is a strictly positive combination of the generators, and therefore zero is in the relative interior.  Since $W_{(A,B)}^\vee$ is a linear space, zero is a positive combination of its generators,
 \[ 0 = \sum_{j=1}^n \sum_{i \in S_j} c_{i,j}a_j^Te_i^T - \sum_{j=1}^m \sum_{i \in T_j} d_{i,j}e_ib_j^T \]
 where $S_j$ is the set of zeros in $a_j$ and $T_j$ the set of zeros in $b_j$.
 Let $v$ denote the same positive combination of the corresponding generators of $W_{(A',B')}^\vee$,
  \[ v = \sum_{j=1}^n \sum_{i \in S_j} c_{i,j}(a_j^Te_i^T + a_{j,r+1}e_{r+1}e_i^T) - \sum_{j=1}^m \sum_{i \in T_j} d_{i,j}e_ib_j^T \]
  \[= \sum_{j=1}^n \sum_{i \in S_j} c_{i,j}a_{j,r+1}e_{r+1}e_i^T. \]
 The matrix $v$ is strictly positive on the first $r$ entries of the last row and zero elsewhere, and its positive entries depend on the new positive entries chosen for $A'$.  The convex cone
  $\cone(b_1,\ldots,b_m) \subseteq \B R^r$
 is full dimensional and contained in the positive orthant.
 Choose a vector $w$ in the interior of the cone, so it can be expressed as a strictly positive combination of the columns of $B$.
 We choose the entries $a_{1,r+1},\ldots,a_{n,r+1}$ so that
  \[ \sum_{j=1}^n \sum_{i \in S_j} c_{i,j}a_{j,r+1}e_i = w. \]
 Then $-e_{r+1}w^T$ is a positive combination of the generators of $W_{(A',B')}^\vee$ of the form $-e_{r+1}b_j^T$ and
  \[ v -e_{r+1}w^T = 0. \]
 Thus zero is a positive combination of all the generators.
 
 Finally, to conclude that $(A',B')$ is partially infinitesimally rigid, we show that $W_{(A',B')}$ is contained in $F$.  After modding out by the diagonal scaling directions, $W_{(A',B')}$ is equal to $V_S$, so its elements square to zero.
\end{proof}

An non-trivial algebraic boundary component of $\C M^{m \times n}_{r}$ consisting of matrices with infinitesimally rigid factorizations $(A,B)$ is defined by $r^2-r+1$ zero conditions on $(A,B)$.  The above construction gives a recipe to produce non-trivial algebraic boundary components consisting of matrices with partially infinitesimally rigid decompositions $(A,B)$ that is also defined by zero conditions on $(A,B)$.  However, the number of zero conditions is generally fewer.  On the other hand, each matrix has a higher dimensional space of nonnegative factorizations.

The following example demonstrates a matrix and its partially infinitesimally rigid factorization on the algebraic boundary of $\C M^{m \times n}_{4}$.  While infinitesimally rigid factorizations for rank 4 have at least 13 zeros, this example has only 10.  The space of nonnegative factorizations in its neighborhood after modding out by diagonal scaling is 3 rather than zero.

\begin{example}\label{ex:pir}
 Let $M \in \C M^{4 \times 3}_{3}$ be the matrix with nonnegative factorization
  \[ A = \begin{pmatrix} 0&1&2\\ 1& 0&2\\ 2&1& 0 \\ 1&2& 0 \end{pmatrix},
  \quad
  B = \begin{pmatrix} 0&1&1\\ 1& 0&1\\ 1&1& 0 \end{pmatrix}.\]
 It can be checked that $W_{(A,B)}$ consists only of the diagonal, so this factorization is infinitesimally rigid.  While $M$ is not on the boundary because it has only 3 columns, it could be expanded into a boundary instance by adding postive rows and columns per Lemma \ref{lem:extend}.
 
 We apply the construction of Proposition \ref{prop:pir} to get $M' = A'B'$ with
  \[ A' = \begin{pmatrix} 0&1&2&1\\ 1& 0&2&1\\ 2&1& 0&1 \\ 1&2& 0&2 \end{pmatrix},
  \quad
  B' = \begin{pmatrix} 0&1&1&1\\ 1& 0&1&1\\ 1&1& 0&1 \\  0& 0& 0&1 \end{pmatrix}.\]
 It can be checked that modulo the diagonal, $W_{(A',B')}$ is the space of matrices supported on entries $(1,4),(2,4),(3,4)$, so the factorization is partially infinitesimally rigid.  $M'$ is also not on the boundary of $\C M^{4 \times 4}_{4}$ but can be expanded into a boundary instance with the same zero pattern.  In the space of nonnegative factorizations of $M'$ in a neighborhood of $(A',B')$, the last column of $A'$ has full dimensional freedom, while the other entries are fixed except for the diagonal action.  The variation of the last column of $A'$ varies the last column of $B'$ while the other entries of $B'$ are also unchanged.
 
 The geometric picture of nested polytopes, $P' \subseteq \Delta' \subseteq Q'$, for $(A',B')$ is shown in Figure \ref{figure:tetra}.  The 3-simplex $\Delta'$ shares a facet $\Delta$ with $P'$.  Slicing along the affine span of $\Delta$ recovers the nested polygons $P \subseteq \Delta \subseteq Q$ associated to $(A,B)$.  The facet $\Delta$ of $\Delta'$ is locked in place by this lower dimensional configuration.  On the other hand, the vertex of $\Delta'$ opposite $\Delta$ is locally free to move in a 3-dimensional neighborhood of its position.
 \end{example}
 
 \begin{figure}
\centering
\begin{subfigure}[b]{0.6\textwidth}
\includegraphics[height=5cm]{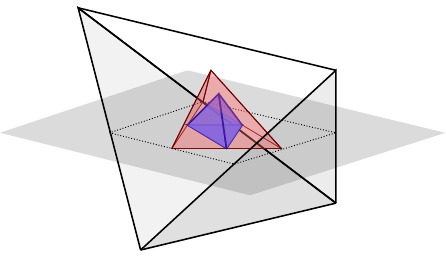}
\caption{$P' \subseteq \Delta' \subseteq Q'$}
\end{subfigure}
\begin{subfigure}[b]{0.35\textwidth}
\includegraphics[height=5cm]{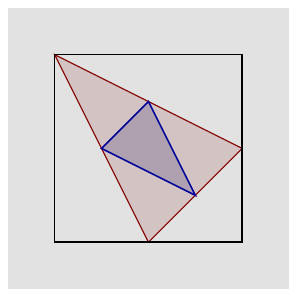}
\caption{$P \subseteq \Delta \subseteq Q$}
\end{subfigure}
\caption{In Example~\ref{ex:pir}, slicing the nested polytopes for $(A',B')$ along the hyperplane of facet $\Delta$ of produces the nested polytopes for $(A,B)$.}
\label{figure:tetra}
\end{figure}
 
 We note that in Figure 8 of \cite{mond2003stochastic} Mond, Smith and van Straten allude to the existence of configurations like Example \ref{ex:triangles}, but they do not ellaborate further on their properties or construction.

\begin{question}
 Are there non-trivial boundary components of $\C M^{m \times n}_{r}$ for $r \geq 4$ consisting of matrices with non-isolated partially infinitesimally rigid factorizations that do not come from the construction of Proposition \ref{prop:pir} or its dual?
\end{question}

\section{Symmetric matrices and completely positive rank}\label{sec:symmetric}

In this section we adapt our results to the case of symmetric matrices.  Let $M$ be an $n\times n$ real nonnegative symmetric matrix.  The {\em completely positive rank} of $M$, denoted $\cprank M$, is the smallest $r$ such that $M = AA^T$ for some nonnegative $n\times r$ matrix $A$~\cite{abraham2003completely}.  For fixed $r$ and $n$ we examine the set of symmetric matrices $M$ with rank and completely positive rank both equal to $r$, as a subset of the symmetric $n\times n$ matrices of rank $r$.

Let $M$ have rank $r$, and $A$ be a symmetric rank factor of $M$, meaning that $M = AA^T$ and $A$ is a $n\times r$ matrix.  The set of all symmetric rank factors of $M$ is
 \[ \{AC \mid C \in \Orth(r)\} \]
where $\Orth(r)$ is the orthogonal group on $\R^r$, which consists of the matrices $C$ that satisfy $C^{-1} = C^T$.  Fixing $A$, we can then identify $\Orth(r)$ with the set of rank factors of $M$ by the linear map $C \mapsto AC$.  Let
  \[ U_A = \{D \in \R^{r\times r} \mid AD \geq 0\}. \]
$M$ has completely positive rank $r$ if and only if $\Orth(r) \cap U_A$ is not empty.

Now we suppose that $\cprank M = r$.  To understand the rigidity of a nonnegative factor $A$, we study $C$ in a neighborhood of $I \in \Orth(r)$ that satisfy $AC \geq 0$.  The infinitesimal motions of $A$ consist of the tangent directions such that
 \begin{eqnarray}
&a^T_i \dot{a}_j + \dot{a}^T_i a_j = 0 \text{ for } (i,j) \in [n] \times [n], \label{eqn:sym_rigidity1}\\ 
&a_i + t\dot{a}_i \geq 0 \text{ for } i \in [n] \text{ and } t \in [0,\epsilon) \label{eqn:sym_rigidity2}
\end{eqnarray}
The tangent space of $\Orth(r)$ at $I$, denoted $T_I\Orth(r)$, consists of all skew symmetric $r \times r$ matrices $D$, which we identify with $\R^{\binom{r}{2}}$ by the coordinates above the diagonal, $d_{ij}$ with $i < j$.  The directions that satisfy Equations \ref{eqn:sym_rigidity1} are $AD$ where $D \in T_I\Orth(r)$.  Let $W_A \subseteq T_I\Orth(r)$ be the cone of tangent directions such that $A(I+tD)$ is an infinitesimal motion.  As in the nonsymmetric case, $W_A$ is cutout by linear inequalities coming from the zero entries of $A$.  If row $a_i^T$ has a zero in entry $j$, it imposes condition $d_j^Ta_i \geq 0$ where $d_j$ is the $j$th column of $D$.
Unlike in the nonsymmetric case, a factor $A$ of $M$ has no trivial deformations, so we have the following definitions.

\begin{definition}
 A nonnegative factor $A$ is {\em locally rigid} if it is an isolated solution to $M = AA^T$ and $A \geq 0$.  $A$ is {\em infinitesimally rigid} if it has no infinitesimal motions.
\end{definition}

All of the theorems from Section \ref{sec:rigid} have analogous statements for symmetric matrices and completely positive rank.  The corresponding results follow. 

\begin{proposition} \label{theorem:infinitesimally_rigid_factorizations_CP}
 $A$ is infinitesimally rigid if and only if $W_A^\vee \cong \R^{\binom{r}{2}}$.
\end{proposition}  

\begin{theorem}\label{prop:SWpoint}
If $A$ is an infinitesimally rigid nonnegative rank-$r$ factor then
  \begin{itemize}
  \item $A$ has at least $(r^2-r)/2 + 1$ zeros and
  \item for every distinct pair $i, j$ taken from $1,\ldots,r$, there must be a row of $A$ with a zero in position $i$ and not in position $j$.
 \end{itemize}
\end{theorem}
\begin{proof}
If $W_A^\vee \cong \R^{\binom{r}{2}}$, then it must have at least $\binom{r}{2}+1$ cone generators.  The generators of $W_{A}^\vee$ are in bijection with the zeros in $A$.
 
For each coordinate $d_{ij}$ with $i < j$ there must be at least one generator of $W_A^\vee$ with a strictly positive value there, and one with a strictly negative value.  Since $A$ is nonnegative, to get a positive value in coordinate $d_{ij}$ requires $A$ to have a row with zero in the $j$th entry and a positive value in the $i$th entry.  To get a negative value requires $A$ to have a row with zero in the $i$th entry and a positive value in the $j$th entry, since $d_{ji} = -d_{ij}$.  
\end{proof}

\begin{corollary} \label{cor:sym_positive}
If $A$ is an infinitesimally rigid nonnegative rank-$r$ factor with exactly $(r^2 - r)/2+1$ zeros, then $M$ is strictly positive.
\end{corollary}

\begin{proof}
If $A$ is infinitesimally rigid, then the dual cone $W^{\vee}_A \cong \R^{(r^2-r)/2}$.  If $A$ has only $(r^2 - r)/2+1$ zeros, the corresponding generating set of $W^{\vee}_A$ is minimal, so the only linear relation among the generators must be among all $(r^2-r)/2$.

If $AA^T$ has a zero in entry $ij$ then rows $a_i$ and $a_j$ of $A$ have zeros in complementary positions so that $a_i a_j^T = 0$.  Since the support of $a_j$ is contained in the set of columns for which $a_i$ is zero, the outer product matrix $a_i^Ta_j$ can be expressed as a nonnegative combination of the dual vectors coming from $a_i$.  Similarly, the matrix $-a_i^Ta_j$ can be expressed as a nonnegative combination of the dual vectors coming from $a_j$.  Summing these gives a linear relation among a strict subset of the generators, which is a contradiction.
\end{proof}

\begin{corollary}\label{lemma:sym_row}
If $A$ is an infinitesimally rigid nonnegative factor, then there is at least one zero in every column of $A$.
\end{corollary}

\begin{corollary}
If $M$ is strictly positive and $A$ is an infinitesimally rigid nonnegative rank-$r$ factor of $M$, then there are at most $r-2$ zeros in every row of $A$.
\end{corollary}

\begin{proof}
Since $M$ is positive, no row of $A$ can contain only zeros.
If a row of $A$ contains $r-1$ zeros, then there is a column of $A$ that does not contain any zero, because otherwise $AA^T$ would have a zero entry. This contradicts Corollary~\ref{lemma:sym_row}. 
\end{proof}

\begin{lemma}\label{lemma:sym_entries}
If $A$ is an infinitesimally rigid nonnegative rank-$r$ factor with $(r^2-r)/2+1$ zeros, then there are at most $r-1$ zeros in every column of $A$.
\end{lemma}

\begin{proof}
As in the proof of Corollary \ref{cor:sym_positive}, the only linear relation among the generators of $W^{\vee}_{(A,B)}$ must be among all $(r^2-r)/2+1$ generators.  If there were $r$ zeros in the same column of $A$, then there would be $r$ generators of $W^{\vee}_{A}$ contained in a $r-1$ dimensional subspace, implying a smaller linear relation which is impossible.
\end{proof}

\begin{lemma}\label{lemma:sym_zero rectangles}
 Let $A$ be an infinitesimally rigid nonnegative rank-$r$ factorization with $(r^2-r)/2+1$ zeros.  Let $\alpha\subseteq [r]$ and suppose $A$ has a $k \times |\alpha|$ submatrix of zeros with columns $\alpha$.  Then
  \[ k \leq (r-|\alpha|). \]
\end{lemma}

\begin{proof}
As in the proof of Corollary \ref{corollary:infinitesimally_rigid_positive}, a generating set of size $(r^2-r)/2+1$ is minimal, so the only linear relation among the generators must be among all of them.  It can be checked that the zeros of $A$ described above correspond to $k|\alpha|$ generators of $W^{\vee}_{A}$ supported on entries $([r] \setminus \alpha) \times \alpha$. The number of generators cannot exceed the number of entries they are supported on, which gives the inequality
\[ k|\alpha| \leq (r-|\alpha|)|\alpha|. \]
\end{proof}

Let $Z_A$ be the matrix whose columns are the generators of $W_A^V$ and let $c$ be the number of columns of this matrix. 

\begin{proposition} \label{prop:generic_locally_rigid_is_infinitesimally_rigid_CP}
Let $M$ be a rank-$r$ matrix. If $AA^T$ is a size-$r$ completely positive factorization of $M$ that is locally rigid but not infinitesimally rigid, then $\text{K-}rank(Z_A)<\min(c,\binom{r}{2})$. 
\end{proposition}

The proof of this proposition is analogous to the proof of Proposition~\ref{prop:generic_locally_rigid_is_infinitesimally_rigid}.
We remark that most other conclusions of Sections \ref{sec:isolated} and \ref{sec:rigidity_and_boundaries} can also be extended to the symmetric case, but we leave this to the enterprising reader.

Proposition~\ref{prop:generic_locally_rigid_is_infinitesimally_rigid_CP} together with Proposition~\ref{theorem:infinitesimally_rigid_factorizations_CP} gives Algorithm~\ref{algorithm:local_rigidity_CP} for determining infinitesimal and local rigidity of a nonnegative matrix factorization.

\begin{algorithm}
\caption{Local rigidity of a size-$r$ completely positive factorization given by $AA^T$}\label{algorithm:local_rigidity_CP}
\begin{algorithmic}[1]
\Procedure{LocalRigidityCPF}{$A,r$}
   \State Construct the matrix $Z_{A}$. Let $c$ be the number of columns of $Z_{A}$.
   \If{the Kruskal-rank of $Z_{A}$ is equal to $\min(c,\binom{r}{2})$}
       \State construct the polyhedral cone $W^V_{A}$ spanned by the columns of $Z_{A}$.
       \If{$W^V_{A}$ is isomorphic to $\R^{\binom{r}{2}}$}
          \State \textbf{return} $AA^T$ is infinitesimally and locally rigid.
       \Else
          \State \textbf{return} $AA^T$ is neither infinitesimally nor locally rigid.
       \EndIf     
   \Else
      \State \textbf{return} $AA^T$ is not infinitesimally rigid; local rigidity cannot be determined.
   \EndIf
\EndProcedure
\end{algorithmic}
\end{algorithm}

\appendix
\section{Infinitesimally rigid factorizations for $5 \times 5$ matrices of nonnegative rank four} \label{sec:realizations_of_infinitesimally_rigid_factorizations}

In this appendix, we will present infinitesimally rigid factorizations for $5 \times 5$ matrices with positive entries and of nonnegative rank four. In particular, we will show that for every zero pattern with $13$ zeros satisfying the conditions of Theorem~\ref{prop:Wpoint}, there exists an infinitesimally rigid nonnegative factorization realizing this zero pattern.

We consider zero patterns up to the action that permutes the rows of $A$, simultaneously permutes the columns of $A$ and the rows of $B$, permutes the rows of $B$ and transposes $AB$. As the first step, we use \texttt{Macaulay2}~\cite{M2} to construct an orbit representative under this action for all zero patterns with $13$ zeros satisfying the  conditions of Theorem~\ref{prop:Wpoint}. There are 15 such orbit representatives.

Then for every zero pattern we construct random realizations by choosing non-zero entries uniformly at random between $1$ and $1000$.
Finally, we use \texttt{Normaliz}~\cite{Normaliz} to find nonnegative factorizations that are infinitesimally rigid based on Definition~\ref{def:rigid}. For each of the $15$ zero patterns, we are able to construct an infinitesimally rigid realization:

\begin{scriptsize}
$$
\begin{pmatrix}
104184 & 229176 & 94392 &  336996 & 77040\\
94663 & 117528 & 485070 & 3404 &   7979\\
535318 & 168896 & 1169348 & 255210 & 182576 \\
156494 & 310908 & 1119179 & 316225 &  460213 \\
763917 & 337540 & 876372 &  1016103 & 574666
\end{pmatrix}
=
\begin{pmatrix}
0 & 0 & 396 & 108\\
0 & 0 & 4 & 555\\
0 & 470 & 0 & 812\\
455 & 0 & 0 & 926\\
194 & 761 & 550 & 0
\end{pmatrix}
\begin{pmatrix}
0 & 260 & 681 & 695 & 985\\
847 & 0 & 978 & 543 & 366\\
217 & 522 & 0 & 851 & 191\\
169 & 208 & 874 & 0 & 13
\end{pmatrix}
$$
$$
\begin{pmatrix}
210729&402419& 94831&122655&193579\\
242132&124696&781275&579876&739205\\
618738&197370&434676&846486&1143228\\
50400&233301&221994& 60009& 34134\\
107007& 33966&457653&315558&360201
\end{pmatrix}
=
\begin{pmatrix}
  0&  0&221&407\\
  0&764&  0&143\\
  0&444&918&  0\\
249&  0&  0&225\\
189&336& 27&  0
\end{pmatrix}
\begin{pmatrix}
  0&149&681&241& 91\\
275&  0&979&759&958\\
541&215&  0&555&782\\
224&872&233&  0& 51
\end{pmatrix}.
$$
$$
\begin{pmatrix}
573705&806520&167622&246500&531659\\
397096& 39600&299176& 63720&274120\\
131646&403260& 30269&226915&264510\\
  9114& 85160&311182&827468&851798\\
147857&  3200&351037&599025&697755
\end{pmatrix}
=
\begin{pmatrix}
  0&  0&425&921\\
  0&472&  0& 80\\
  0&  1&391&163\\
862&  0& 98&  0\\
640&199&  0&  0
\end{pmatrix}
\begin{pmatrix}
  0&  5&361&894&927\\
743&  0&603&135&525\\
 93&825&  0&580&538\\
580&495&182&  0&329
\end{pmatrix}
$$
$$
\begin{pmatrix}
30893 & 319912 & 149770 & 873 & 111428\\ 
383490 & 87990 & 5580 & 628440 & 587250\\ 
560076 & 1030324 & 331070 & 288045 & 350647\\ 
203830 & 305184 & 277512 & 264376 & 205933\\ 
90911 & 142936 & 500784 & 618842 & 609633
\end{pmatrix}
=
\begin{pmatrix}
0 &  0 &  356 & 9\\
0 &  870 &  0 &  30\\
0 &  302 & 469 & 731\\
403 & 0 &  0 &  374\\
852 & 190 & 147 & 0
\end{pmatrix}
\begin{pmatrix}
0 & 0 & 516 & 566 & 511\\
422 & 73 &  0 &  719 & 675\\
73 &  878 & 416 & 0 &  313\\
545 & 816 & 186 & 97 &  0
\end{pmatrix}.
$$
$$
\begin{pmatrix}
553924 & 99854 & 348351 & 183860 & 20114 \\ 401268 & 3372 & 802602 & 250881 & 155672 \\ 1091328 & 648606 & 538803 &176341 & 151574 \\ 472277 & 506248 & 136080 & 591292 & 591056 \\ 377978 & 477454 & 470565 & 322776 & 461574
\end{pmatrix}
=
\begin{pmatrix}
0 &0 &113 &634 \\ 0 &671 &0 &562 \\ 0 &71 &759 &576 \\ 697 &0 &0 &270 \\ 346 &520 &267 &0
\end{pmatrix}
\begin{pmatrix}
401 &724 &0 &736 &848 \\ 0 &0 &774 &131 &232 \\ 896 &850 &255 &0 &178 \\ 714 &6 &504 &290 &0
\end{pmatrix}
$$
$$
\begin{pmatrix}
292425 & 60900 & 31581 & 170931 & 7358\\
8056 & 89782 & 548546 & 684912 & 505520\\
98680 & 758632 & 1234092 & 742008 & 1123962\\
428876 & 6358 & 306000 & 865802 & 851174\\
888312 & 823270 & 758974 & 620872 & 1215638
\end{pmatrix}
=
\begin{pmatrix} 
0& 0& 525& 13\\ 
0& 106& 0& 751\\ 
0& 888& 56& 795\\ 
578& 0& 0& 500\\ 
568& 866& 720& 0
\end{pmatrix}
\begin{pmatrix}
742 & 11 & 0 & 709 & 983\\
76 & 847 & 839 & 0 & 759\\
557 & 116 & 45 & 303 & 0\\
0 & 0 & 612 & 912 & 566
\end{pmatrix}.
$$
$$
\begin{pmatrix}
 348984& 214425& 353658&  81504& 608634\\
 333621&  42811& 108265& 141389&  79520\\
 457700&   5980& 467723& 866662& 841426\\
  91308& 220419& 483054& 706686&1353778\\
 342940& 384918& 120318& 550726& 945556\\
\end{pmatrix}
=
\begin{pmatrix}
  0&  0&867&288\\
  0&112&  0&295\\
937&  0&  0&460\\
832&102&761&  0\\
110&898&298&  0
\end{pmatrix}
\begin{pmatrix}
  0&  0&319&786&898\\
358&348&  0&517&710\\
 72&243&286&  0&702\\
995& 13&367&283&  0
\end{pmatrix}
$$
$$
\begin{pmatrix}
 88076&294646&658787&902872&244559\\
  2216&  4216&596705&652698&250465\\
279360&180864&769506 &1051380&391634\\
553284&826606&765406&293965&883775\\
696039&897917&148301&832169&169525
\end{pmatrix}
=
\begin{pmatrix}
  0&  0&454&713\\
  0&  8&  0&711\\
288&  0&  0&926\\
239&998&232&  0\\
541& 37&830&  0
\end{pmatrix}
\begin{pmatrix}
970&628&  0&699&257\\
277&527&733&  0&824\\
194&649&146&547&  0\\
  0&  0&831&918&343
\end{pmatrix}
$$
$$
\begin{pmatrix}
948201&723609&958755&591858&397953\\
222448&218040& 30429&348793& 15825\\
329588&  7189&623001& 12012&469185\\
467424&160704&115092&835504&343912\\
1114797&932972&975775&997164&636096
\end{pmatrix}
=
\begin{pmatrix}
  0&  0&867&753\\
  0&211&  0&189\\
429&  0&553&  0\\
556&864&  0&  0\\
552&270&738&923
\end{pmatrix}
\begin{pmatrix}
  0&  0&207& 28&502\\
541&186&  0&949& 75\\
596& 13&966&  0&459\\
573&946&161&786&  0
\end{pmatrix}
$$
$$
\begin{pmatrix}
264293& 89201&411390& 21016& 54492\\
255674&383544&693861&252463&211653\\
212205&  6665&216806&  6450&103802\\
469696&393840&450523&564374&956188\\
288927&197161&105742&300945&433801
\end{pmatrix}
=
\begin{pmatrix}
  0&  0&239&284\\
  0&351&  0&893\\
 86&  0&215&  0\\
598&954&  0&175\\
154&545& 31&  0
\end{pmatrix}
\begin{pmatrix}
 0&  0&526& 75&637\\
474&360&  0&531&603\\
987& 31&798&  0&228\\
100&288&777& 74&  0
\end{pmatrix}
$$
$$
\begin{pmatrix}
  3230&104329&410573&875858&188790\\
 22527& 66939&204273& 81606& 13419\\
123988& 34611& 82056&713192&305348\\
596448&338171&559708&395192&624199\\
     1460035&246567&270382&584688 &1302924
\end{pmatrix}
=
\begin{pmatrix}
  0&  0&870&323\\
  0& 21&  0&201\\
139&  0&789&  0\\
623& 36&  0&556\\
639&911&480&  0
\end{pmatrix}
\begin{pmatrix}
892&249&  0&272&965\\
977& 96&242&  0&639\\
  0&  0&104&856&217\\
 10&323&991&406&  0
\end{pmatrix}
$$
$$
\begin{pmatrix}
64244 & 119613 & 501370 & 37843 & 259408\\
85315 & 371265 & 69495 & 801995 & 33660\\
83956 & 5004 & 737712 & 957860 & 230056\\
46287 & 566084 & 451221 & 397664 & 269200\\
144598 & 34999 & 923447 & 1330101 & 293244
\end{pmatrix}
=
\begin{pmatrix}
0 & 0 & 523 & 41\\
0 & 510 & 0 & 565\\
772 & 0 & 0 & 556\\
64 & 656 & 417 & 0\\
853 & 13 & 77 & 901
\end{pmatrix}
\begin{pmatrix}
0 & 0 & 867 & 576 & 298\\
0 & 718 & 0 & 550 & 66\\
111 & 228 & 949 & 0 & 496\\
151 & 9 & 123 & 923 & 0
\end{pmatrix}
$$
$$
\begin{pmatrix} 310392& 195156& 317952& 492156& 169188\\
  82320& 581120&  90160& 709152&  19024\\
 519783& 180720&1398418&  74387& 728134\\
  70245& 244363& 505935& 527965& 176138\\
 451143& 501811& 582768& 158964& 396949
\end{pmatrix}
=
\begin{pmatrix}
  0&  0&276&756\\
  0&656&  0&784\\
901&  0&619& 16\\
440&202&  0&669\\
135&493&539&  0
\end{pmatrix}
\begin{pmatrix}
  0&  0&975& 71&387\\
  0&703&  0&303& 29\\
837&288&837&  0&613\\
105&153&115&651&  0
\end{pmatrix}
$$
$$
\begin{pmatrix}
72200& 697140& 19076& 191446& 252354\\
341204& 824131& 90064& 90804& 450580\\
292600& 86846& 319858& 425581& 57573\\
493288& 887466& 592538& 286784& 604086\\
809126& 281001& 625050& 719417& 276676
\end{pmatrix}
=
\begin{pmatrix}
0&0&76&822\\
0&433&0&644\\
490&0&308&79\\
934&626&0&570\\
831&377&539&0
\end{pmatrix}
\begin{pmatrix}
0&0&495&221&68\\
788&651&208&0&584\\
950&66&251&994&0\\
0&842&0&141&307
\end{pmatrix}
$$
$$
\begin{pmatrix}
279265&274840&187355&655433&214052\\
270970& 68600&734264   &  1018514& 89856\\
341531&544696&235555&187012&948873\\
417526&121556&855865&841310&486784\\
 15933&287113&730363&580464&439746
\end{pmatrix}
=
\begin{pmatrix}
  0&  0&236&707\\
  0&702&  0&686\\
849&  0&507&136\\
684&725&  0&470\\
 47&914&326&  0
\end{pmatrix}
\begin{pmatrix}
339&109&235&  0&576\\
  0&  0&787&588&128\\
  0&865&  0&132&907\\
395&100&265&883&  0
\end{pmatrix}
$$
\end{scriptsize}

We conjecture that the $15$ nonnegative factorizations above are in fact globally rigid based on the following evidence. For each of the $15$ matrices above we ran the program by Vandaele, Gillis, Glineur and Tuyttens~\cite{vandaele2016heuristics} with the simulated annealing heuristic ``sa'' ten times. Each run consisted of at most ten attempts and the target precision was $10^{-15}$. In $13$ out of the $15$ cases at least one out of ten runs would find a nonnegative factorization of size four. Each time when a size-four nonnegative factorization was found, it was the same as the original factorization. On average  $6.3$ runs were successful finding a size-four nonnegative factorization. For the third and ninth matrix none of the runs found a size-four nonnegative factorization. The algorithm was much slower for the eighth matrix than for any other matrix in the list. Although only three runs found a nonnegative factorization of target precision in this case, all other solutions looked similar to the original solution as well. This suggests that the algorithm converges slowly for this matrix. In summary, in each of the cases, either this program could not find a nonnegative factorization of target precision or it would find the nonnegative factorization that we started with. If these matrices would have other nonnegative factorizations, we find it unlikely that this would be the case.

 Vandaele, Gillis, Glineur and Tuyttens discuss in~\cite[Section 2]{vandaele2016heuristics} that $A$ and $B$ positive are known to increase the number of factorizations of $AB$ and hence factoring $AB$ is usually easier. This suggests that matrices with small factorization spaces are the most difficult for exact nonnegative matrix factorization algorithms. Hence one application of the $15$ matrices above could be as benchmark matrices for nonnegative matrix factorization algorithms.

We also constructed orbit representatives for all zero patterns with $13$ zeros satisfying the conditions of Theorem~\ref{prop:Wpoint} and Lemma~\ref{lemma:max_number_entries} for larger matrices such that every row of $A$ contains a zero, and the number of columns of $B$ is five or every column of $B$ contains a zero. The number of such zero patterns for each matrix size is listed in Table~\ref{table:number_of_zero_patterns}.
\begin{table}[H]
\caption{Number of zero patterns with $13$ zeros satisfying the conditions of Theorem~\ref{prop:Wpoint} and Lemma~\ref{lemma:max_number_entries} for different matrix sizes such that every row of $A$ contains a zero, and the number of columns of $B$ is five or every column of $B$ contains a zero. }
 \label{table:number_of_zero_patterns}
\begin{center}
\begin{tabular}{ |c| c| c| c| c| c| c| }
\hline
 $5 \times 5$ & $6 \times 5$ & $6 \times 6$ & $7 \times 5$ & $7 \times 6$  & $8 \times 5$ &  $9 \times 5$\\
 \hline
 15 & 26 & 14 & 24 & 11 & 10 & 2\\
 \hline 
\end{tabular}
\end{center}
\end{table}

Differently from the $5 \times 5$ case, not all of these zero patterns automatically satisfy the necessary condition in Lemma~\ref{lemma:zero rectangles} which is more difficult to check than the necessary conditions in Theorem~\ref{prop:Wpoint} and Lemma~\ref{lemma:max_number_entries}. In the case of $6\times 5$ matrices, one out of $26$ zero patterns fails the necessary condition in Lemma~\ref{lemma:zero rectangles}. It is
$$
\begin{pmatrix}
0 & 0 & \cdot & \cdot \\
0 & 0 & \cdot & \cdot \\
0 & \cdot & \cdot & \cdot \\
\cdot & 0 & \cdot & \cdot \\
\cdot & \cdot & 0 & \cdot \\
\cdot & \cdot & \cdot & 0
\end{pmatrix}
\begin{pmatrix}
\cdot & \cdot & 0 & \cdot & \cdot \\
\cdot & \cdot & \cdot & 0 & \cdot \\
0 & 0 & \cdot & \cdot & \cdot \\
\cdot & \cdot & \cdot & \cdot & 0
\end{pmatrix}.
$$
For the rest of the $25$ zero patterns, Huanhuan Chen constructs infinitesimally rigid realizations in his Master's thesis~\cite{chen2019infinitesimally}. He also shows that for larger factorizations there does not exist an infinitesimally rigid factorization realizing  every pattern of $r^2-r+1$ zeros satisfying the conditions of Theorem~\ref{prop:Wpoint} and Lemma \ref{lemma:zero rectangles}. He gives a stronger necessary condition for an infinitesimally rigid realization to exist and conjectures that this condition is also sufficient.

\bigskip

\begin{small}

\noindent
{\bf Acknowledgments.}
We thank two anonymous referees, Huanhuan Chen, Nicolas Gillis, Ivan Izmestiev and Arnau Padrol for helpful comments and suggestions. This material is based upon work supported by the National Science Foundation under Grant
 No.~DMS-1439786 while the authors were in residence at the Institute for Computational and Experimental Research in Mathematics in Providence, RI, during the Fall 2018 semester. Kubjas was supported by the European Union's Horizon 2020 research and innovation programme (Marie Sk\l{}odowska-Curie grant agreement No.~748354, research carried out at LIDS, MIT and PolSys, LIP6, Sorbonne Universit\'e). We acknowledge the computational resources provided by the Aalto Science-IT project.

\smallskip

\end{small}

\bibliography{bibl}
\bibliographystyle{plain}

\smallskip

\smallskip

\noindent Authors' affiliations:

\vspace{0.2cm}
\noindent Robert Krone, Department of Mathematics, University of California, Davis,\newline \texttt{rckrone@ucdavis.edu}

\vspace{0.2cm}
\noindent Kaie Kubjas, Department of Mathematics and Systems Analysis, Aalto University,\\ \texttt{kaie.kubjas@aalto.fi}

\end{document}